\documentclass[10pt,a4paper]{article}
\usepackage[english]{babel}
\usepackage{amssymb}
\usepackage{amsmath}
\usepackage{indentfirst}
\usepackage[dvips]{graphicx}
\usepackage{epsfig}
\usepackage{xcolor}
\usepackage{color}
\usepackage{amsthm} 
\usepackage{comment}
\usepackage{multirow} 
\usepackage{amsfonts}
\usepackage{mathtools} 
\usepackage{setspace} 
\usepackage[colorlinks,citecolor=blue,linkcolor=red]{hyperref}

\newtheorem{cor}{Corollary}[section]

\newtheorem{lemma}[cor]{Lemma}
\newtheorem{teo}[cor]{Theorem}
\newtheorem{algo}{Algorithm}[section]
\newtheorem{rem}[cor]{Remark}

\newtheorem{exm}[cor]{Example}
\newtheorem{assum}[cor]{Assumption}

\newcommand{\R}{\mathbb{R}}

\newcommand{\err}{{\rm err}}
\newcommand{\eoc}{{\rm eoc}}
\newcommand{\eps}{\varepsilon}

\numberwithin{equation}{section}

\begin{document}

\title{Convergence of a scheme for elastic flow with tangential mesh movement}

\author{Paola Pozzi 
\thanks{Fakult\"at f\"ur Mathematik, Universit\"at Duisburg-Essen, 
Thea-Leymann-Stra\ss e 9, 
45127 Essen, Germany, 
\url{paola.pozzi@uni-due.de}} 
\and 
and Bj\"orn Stinner 
\thanks{Mathematics Institute, 
Zeeman Building, 
University of Warwick, 
Coventry CV4 7AL, 
United Kingdom, 
\url{Bjorn.Stinner@warwick.ac.uk}} 
}

\date{\today}
\maketitle

\begin{abstract}
Elastic flow for closed curves can involve significant deformations. Mesh-based approximation schemes require tangentially redistributing vertices for long-time computations. We present and analyze a method that uses the Dirichlet energy for this purpose. The approach effectively also penalizes the length of the curve, and equilibrium shapes are equivalent to stationary points of the elastic energy augmented with the length functional. Our numerical method is based on linear parametric finite elements. Following the lines of \cite{DD09} we prove convergence and establish error estimates, noting that the addition of the Dirichlet energy simplifies the analysis in comparison with the length functional. We also present a simple semi-implicit time discretization and discuss some numerical result that support the theory. 
\end{abstract}

\bigskip

\tableofcontents

\section{Introduction}

Elastic flow is the $L^2$-gradient flow of the bending energy 
\[
 \mathcal{E}(x) = \frac{1}{2} \int_{0}^{2\pi} |\vec{\kappa}|^2 ds
\]
for the elasticity of a rod that is modelled by a closed (i.e.~periodic) regular smooth curve \cite{truesdell}. Here, $x: [0,2\pi] \to \R^{n}$, $x=x(u)$, stands for its parametrisation, $ds = |x_{u}| du$ its length element, and $\vec{\kappa} = x_{ss} = \tau_s$ its curvature vector where $\tau = \partial_{s} x=\frac{x_{u}}{|x_{u}|}$ is a unit tangent vector. Because of the scaling properties of $\mathcal{E}$ (a circle of radius $R$ has energy $\pi/R \to 0$ as $R \to \infty$), one instead can consider \cite{DKS}
\begin{equation} \label{eq:energylength}
 \mathcal{E}_{\tilde{\lambda}}(x):=\mathcal{E}(x)+ \tilde{\lambda} \mathcal{L}(x)
\end{equation}
for some positive $\tilde{\lambda}$ where 
\[
 \mathcal{L}(x) = \int_{0}^{2\pi} ds= \int_{0}^{2\pi} |x_{u}| du
\]
is the length functional. The corresponding $L^{2}$-gradient flow of $\mathcal{E}_{\tilde{\lambda}}$ is given by
\begin{align}\label{flowDKS}
x_{t} = -\nabla_{s}^{2} \vec{\kappa} -\frac{1}{2} |\vec{\kappa}|^{2} \vec{\kappa} +\tilde{\lambda} \vec{\kappa}. 
\end{align}
Here $\nabla_{s} \phi =\partial_{s} \phi - (\partial_{s} \phi \cdot \tau) \tau$ denotes the normal component of the derivative with respect to arc length for a given vector field $\phi: [0,2\pi] \to \R^{n}$ along the curve. 
Note that the first variation of the energy functional $\mathcal{E}_{\tilde{\lambda}}$ yields a gradient in normal direction and thus does not ``suggest'' any tangential components. A more geometrical formulation (specifying the motion in the normal direction only) is thus given by
\begin{align}\label{gflow}
Px_{t}=x_{t} - (x_{t} \cdot \tau) \tau = -\nabla_{s}^{2} \vec{\kappa} -\frac{1}{2} |\vec{\kappa}|^{2} \vec{\kappa} +\tilde{\lambda} \vec{\kappa} 
\end{align}
where $P = I - \tau \otimes \tau$ with a unit tangent $\tau$ is the projection to the normal space. With other words, in \eqref{flowDKS} the velocity vector is entirely normal to the curve, whereas in the formulation \eqref{gflow} different tangential components are admissible. 

This geometric flow (and its variations) has been thoroughly studied in recent years both numerically and analytically (see for instance \cite{DKS}, \cite{Polden}, \cite{Koiso}, \cite{LangerSinger1}, \cite{Wen}, \cite{Wen2}, \cite{LinSch}, \cite{MPP20}, \cite{BGN12}, \cite{Bartels13}, \cite{OlRu10}, \cite{Poz15}, \cite{Bondarava}, \cite{DD09}, and references given in there).

From a numerical point of view, quantitative error analysis of a FEM scheme has been provided only in the case of \eqref{flowDKS} where the flow has no tangential component (see \cite{DD09}). However, the lack of tangential components has major disadvantages from a computational viewpoint, as the grid might degenerate quickly. 
In \cite{BGN12} the authors work with \eqref{gflow} and add a tangential movement that is beneficial for the mesh quality. 
They provide a FEM scheme for which a stability bound for a continuous in time semi-discrete scheme holds together with an equi-distribution mesh property. 
In \cite{Bartels13} the author considers the case of inextensible curves, therefore the problem of a possible degeneration of the mesh is ruled out by the formulation of the problem itself. A convergence analysis is presented, while a quantitative control of the error is missing. 

Here we are interested in the problem where the curves are permitted to change their length in time.
We would like to achieve two goals: good grid properties and a quantitative numerical error analysis.
Ruling out \eqref{flowDKS} because of the aforementioned grid degeneration problems, one way to tackle the problem would be to 
introduce a tangential component for \eqref{gflow} that is suitable both for analytical and numerical purposes. 
This step is sort of ``artificial'' but the task far from straightforward
because of the following two ``clashing'' points:  
\begin{itemize}
\item 
From an analytical point of view (for instance when showing short-time existence of the flow)
one is interested in tangential components that remove the degeneracy of the differential operator. 
\item 
From a numerical perspective it is usually advantageous to stay as close as possible to the gradient structure of the flow (for instance, for stability) and to derive a variational problem such that only operations are performed that are also admissible in a finite element setting. 
\end{itemize}

Our approach to the problem is to revisit the ideas behind the definition of elastic flow and find a suitable compromise between analysis and numerics.
Thus we propose a new model of elastic flow for which we are able to provide an error analysis of the semi-discrete scheme (similarly to \cite{DD09}) and such that the mesh presents good numerical properties. Analytical questions, such as long-time existence of the flow, are treated in 
a separate work~\cite{Paolapreprint}.  

Let us thus consider again the classical definition of elastic flow.
As noted above the penalization of the bending energy through the length functional is motivated by the fact that we do not wish the curve to flatten out at infinity. 
A penalization of the length as in \eqref{eq:energylength} is effective in this respect. 
There is however another way to penalize growth of the curve that also induces a nice control of the parametrization of the curve. 
Inspired by classical theory on minimal surfaces (\cite{DHS339}), our idea is to replace the length functional $\mathcal{L}(x)$ with the Dirichlet energy 
\begin{align*}
 \mathcal{D}(x):= \frac{1}{2} \int_{0}^{2\pi} |x_{u}|^{2} du 
\end{align*}
whose beneficial effects for the mesh quality in moving mesh methods have been already experienced in different ways, see for instance \cite{DD94,EF17} and related work. Thus, we want to minimize 
\begin{align} \label{newFun}
 \mathcal{D}_{\lambda}(x) := \mathcal{E}(x) + \lambda \mathcal{D}(x)
\end{align}
for some positive given $\lambda$.
Note that $ \mathcal{L}(x) \leq \sqrt{2\pi} \sqrt{2 \mathcal{D}(x)}$, whence $\mathcal{E}_{\lambda}(x)$ is dominated by $ \mathcal{D}_{\lambda}(x)$.
Noting that $\frac{x_{uu}}{|x_u|} \cdot \tau = (|x_u|)_s$, 
the $L^{2}$-gradient flow for $ \mathcal{D}_{\lambda}(x)$ is given by
\begin{align}
x_{t}& = -\nabla_{s}^{2} \vec{\kappa} -\frac{1}{2} |\vec{\kappa}|^{2} \vec{\kappa} +\lambda \frac{x_{uu}}{|x_{u}|} \label{1.2} \\
&= -\nabla_{s}^{2} \vec{\kappa} -\frac{1}{2} |\vec{\kappa}|^{2} \vec{\kappa} +\lambda \vec{\kappa} |x_{u}| + \lambda (|x_{u}|)_{s} \tau. \label{1.2b}
\end{align}

The flow is obviously different from \eqref{flowDKS} and \eqref{gflow} but the set of extremal points for the energies $\mathcal{E}_{\tilde{\lambda}}$ and $\mathcal{D}_{\lambda}$ are the same in the following sense: 
since $\mathcal{D}$ is not invariant under re-parametrization of the domain, the first variation in a stationary point $x$ of $\mathcal{D}_{\lambda}$ with respect to tangential deformations yields that $|x_{u}|=const$, 
so $x$ is parametrized by constant speed $|x_{u}|=:S_{x}$. 
Comparing \eqref{1.2b} with \eqref{flowDKS} we see that $x$ then also is a stationary point of $\mathcal{E} _{\tilde{\lambda}}$ for $\tilde{\lambda} = \lambda S_{x}$. 

Vice versa, if $x$ is stationary point for $\mathcal{E}_{\tilde{\lambda}}$, then its parametrization by constant speed, denoted by $S_{x}$ again for convenience, is a stationary point of $\mathcal{D}_{\lambda}$ with $\lambda = \tilde{\lambda} / S_{x}$. 
We also note that \eqref{1.2} is not a geometric flow: 
indeed, the Dirichtlet functional is not invariant under reparametrizations as the length functional. Its purpose here is to single out a tangential contribution to the parametrization that yields good mesh properties during the evolution.

In Section \ref{sec2} we present a FEM discretization very much in the style and spirit of \cite{DD09}. As a consequence, the main error estimates, see Theorem~\ref{mainteo} below, are analogous to those presented in \cite[Thm.~2.3]{DD09}: more precisely, after considering a mixed formulation of the flow, we prove a linear control with respect to the space grid parameter $h$ for suitable norms of the error of the parametrization $x$ and its curvature vector $\vec{\kappa}$. Notice, however, that the substitution of the length functional with the Dirichlet energy simplifies the analysis in a significant manner (see comments below Theorem~\ref{mainteo}). 
Numerical experiments in Section~\ref{sec:numerics} support the theoretical result. They also highlight the good mesh properties of our scheme in 2D and 3D. We also provide some comparisons with other schemes. Finally, although the Dirichlet energy is a great ally in the numerical analysis for the semi-discrete formulation, it is less clear how it could help in showing stability for the fully discrete scheme that we use in our computations (cp. also \cite{Bondarava} and \cite[Remark 6.4]{Poz15}). This aspect will be subject of future research. 

\bigskip
\noindent \textbf{Acknowledgements:} This project has been funded by the Deutsche Forschungsgemeinschaft (DFG, German Research Foundation)- Projektnummer: 404870139.

\section{Main results}
\label{sec2}

\subsection{Variational problem formulation}

We now consider the $L^{2}$-gradient flow for \eqref{newFun} and are interested in numerically 
approximating the evolution given in \eqref{1.2} for some positive fixed $\lambda>0$.
Before formulating the problem in a variational form that is amenable to a finite element approximation we first make a remark on how self-similar solutions to \eqref{1.2} and \eqref{flowDKS} compare.

\begin{exm}[Self-similar evolution of a circle parametrized by constant speed]\label{esempio}

Let $\lambda,\tilde{\lambda} \in \R^{+}$ and consider functions of the form $x(t,u) = R(t) (\cos u, \sin u)$. A direct computation gives that $x$ solves
\begin{align*}
 x_{t} = -\nabla_{s}^{2} \vec{\kappa} -\frac{1}{2} |\vec{\kappa}|^{2} \vec{\kappa} +\tilde{\lambda} \vec{\kappa}
\end{align*}
if and only if $\dot{R}=\frac{1}{2 R^{3}} -\frac{\tilde{\lambda}}{R} =\frac{1}{R^{3}}\left (\frac{1}{2}-\tilde{\lambda} R^{2} \right)$, 
which has the stationary solution $\tilde{R} = 1 /\sqrt{2\tilde{\lambda}}$. 

On the other hand, $x$ solves 
\begin{align*}
 x_{t} = -\nabla_{s}^{2} \vec{\kappa} -\frac{1}{2} |\vec{\kappa}|^{2} \vec{\kappa} +\lambda \frac{x_{uu}}{|x_{u}|}
\end{align*}
if and only if $\dot{R}=\frac{1}{2 R^{3}} -\lambda = \frac{1}{R^{3}} \left(\frac{1}{2} -\lambda R^{3} \right)$ 
with the stationary solution $\bar{R} = 1 / \sqrt[3]{2 \lambda}$. 
\end{exm}

Since our approach follows very closely the one presented in \cite{DD09}, we use for simplicity the same notation. In particular the curvature vector of $x$ is now denoted by $y$, $$ \vec{\kappa} = y \quad \mbox{in the following.} $$ 
The strong formulation \eqref{1.2} then becomes 
\begin{equation} \label{strongPDE}
  x_{t} = -\nabla_{s}^{2} y -\frac{1}{2} |y|^{2} y + \lambda \frac{x_{uu}}{|x_{u}|}, \quad y = \frac{1}{|x_{u}|} \Big{(} \frac{x_{u}}{|x_{u}|} \Big{)}_{u}. 
\end{equation}
The problem is closed by imposing an initial condition of the form
\begin{equation} \label{init_cond}
 x(0,\cdot) = x_{0}(\cdot) \quad \mbox{on } [0,2\pi)
\end{equation}
for given periodic initial data denoted by $x_{0}$. 
For the error analysis, we make the following assumptions: 
 
\begin{assum} \label{assum}
 We assume that $x_{0}$ is smooth and that \eqref{strongPDE} with \eqref{init_cond} has a smooth, periodic (in space) solution $x: [0,T] \times [0, 2\pi] \to \R^{n}$. We furthermore assume that there exist constants $0 < c_{0} \leq C_{0}$ such that
 \begin{align}\label{3.1}
  c_{0} \leq |x_{u}| \leq C_{0}, \qquad \qquad |y| \leq C_{0} \qquad \text{ in } [0,T] \times [0, 2\pi].
 \end{align}
\end{assum}

\noindent \textbf{Weak formulation.}
Following the same steps presented in \cite[\S~2]{DD09} in the computation of the first variation we consider the following weak form for the gradient flow of $\mathcal{D}_{\lambda}$: 
\begin{align}\label{2.2}
\int_{0}^{2\pi} (x_{t} \cdot \phi ) |x_{u}| - \int_{0}^{2\pi} \frac{P y_{u} \cdot \phi_{u}}{|x_{u}|} -\frac{1}{2}
\int_{0}^{2\pi} |y|^{2} (\tau \cdot \phi_{u} ) +\lambda \int_{0}^{2\pi} x_{u} \cdot \phi_{u} &=0\\\label{2.3}
\int_{0}^{2\pi} (y \cdot \psi) |x_{u}| + \int_{0}^{2\pi} (\tau \cdot \psi_{u} ) &=0
\end{align}
for all $\phi, \psi \in H^{1}_{per}(0, 2\pi)$ and $t \in [0,T]$, together with the initial condition \eqref{init_cond}. Here, $P$ denotes the normal projection matrix $P =I - \tau \otimes \tau$.

A key advantage of the above formulation is that \emph{no integration by parts} has been carried out in the derivation of the first variation. In particular, all steps performed hold also in the finite dimensional spaces that we will choose below.
This allows to recover very easily the equivalent of the energy estimate $\frac{d}{dt} \mathcal{D}_{\lambda}(x) \leq 0$ (see Remark~\ref{rem:energy} below).
Therefore, as mentioned in the introduction, we have stayed true to our principles, choosing a formulation that is close to the variational structure of the problem and that uses operations that are admissible in the discrete space.

\subsection{Spatial discretization}

We use the same notation adopted in \cite{DD09} with the exception that we denote the Euclidean scalar product with $\cdot$ instead of $(\cdot, \cdot)$. Let $0=u_{0} < u_{1} < \ldots <u_{N}= 2\pi$ be a partition of $[0,2\pi]$
into subintervals $I_{j}=[u_{j-1},u_{j}]$. Let $h_{j}:= u_{j} -u_{j-1}$, $h := \max_{j=1, \ldots, N} h_{j}$ and assume
\begin{align*}
h \leq \hat{c} h_{j} \qquad \text{ for all } j=1, \ldots, N
\end{align*}
for some positive $\hat{c}$ independent of $h$. 
Moreover let us assume that
\begin{align}\label{extra-ass}
|h_{j+1}-h_{j}| \leq \bar{c} h^{2} \qquad \text{ for all } j=1, \ldots, N \quad (\text{ and with } \quad h_{N+1}=h_{1}).
\end{align}
for some positive constant $\bar{c}$. 
For approximation purposes we consider the finite element space 
\[ 
 X_{h} = \{ \eta \in C^{0}([0, 2\pi]) \,: \, \eta_{h}\big{|}_{I_{j}} \text{ is linear }, j=1, \ldots, N, \,\text{ and } \eta_{h}(0)=\eta_{h}(2\pi) \} 
\]
with the usual nodal basis functions $\{ \varphi_{1}, \ldots, \varphi_{N} \}$. Let $I_{h}$ denote the standard Lagrange interpolation operator and recall the standard estimates:
\begin{align}\label{2.5}
\| f-I_{h} f\|_{L^{2}} + h \| (f-I_{h} f)_{u}\|_{L^{2}} &\leq Ch^{2} \| f \|_{H^{2}} \qquad \forall \; f \in H^{2}_{per}(0, 2\pi),\\
\label{nostra}
\|(I_{h}f)_{u}\|_{L^{2}} &\leq C \|f _{u} \|_{L^{2}} \qquad \forall \; f \in H^{1}_{per}(0, 2\pi).
\end{align} 
Moreover recall the following useful estimates (\cite[$(2.6),(2.7)$]{DD09})
\begin{align}
\label{2.6}
\int_{I_{j}} |\phi_{h}|^{2} &\leq \int_{I_{j}} I_{h}(|\phi_{h}|^{2}) \leq C\int_{I_{j}} |\phi_{h}|^{2}, \\
\label{2.7}
\int_{I_{j}} \phi_{h} \cdot \psi_{h}& =\int_{I_{j}} I_{h} (\phi_{h} \cdot \psi_{h} ) -\frac{1}{6} h_{j}^{2} \int_{I_{j}} \phi_{hu} \cdot \psi_{hu}
\end{align}
for all $j=1, \ldots, N$ and $\phi_{h}, \psi_{h} \in X_{h}^{n}$. 
In the following analysis we will also make use of the standard inverse estimates
\begin{align}\label{invest}
\|\phi_{hu}\|_{L^{2}} \leq C h^{-1} \| \phi_{h}\|_{L^{2}},\qquad \qquad \|\phi_{hu}\|_{L^{\infty}} \leq C h^{-1} \| \phi_{h}\|_{L^{\infty}}
\end{align}
which hold for $\phi \in X_{h}^{n}$. Next, set
$$ Z_{h} :=\{ z_{h}: [0, 2\pi] \to \R^{n} \, : \, z_{h}|_{I_{j}} \text{ is constant }, j=1, \ldots, N\}$$
and define $Q_{h}: H^{1} ((0, 2\pi), \R^{n}) \to Z_{h}$ by $ (Q_{h} f)|_{I_{j}} := \frac{1}{|I_{j}|} \int_{I_{j}} f$, $j=1, \ldots, N$.
Notice that by \cite[$(2.8)$]{DD09} we have
\begin{align}\label{2.8}
\| f - Q_{h}f \|_{L^{2}} \leq C h \| f \|_{H^{1}} \qquad \forall f \in H^{1} ((0, 2\pi), \R^{n}).
\end{align}

The semi-discrete problem corresponding to \eqref{2.2}, \eqref{2.3} reads as follows: \\
find $x_{h},y_{h}:[0,T] \times [0, 2\pi] \to \R^{n}$ such that
$x_{h}(t, \cdot), y_{h}(t, \cdot) \in X_{h}^{n}$, $0 \leq t \leq T$ and
\begin{align}\label{2.9}
\int_{0}^{2\pi}I_{h} (x_{ht} \cdot \phi_{h} ) |x_{hu}| - \int_{0}^{2\pi} \frac{P_{h} y_{hu} \cdot \phi_{hu}}{|x_{hu}|} 
 -\frac{1}{2}
\int_{0}^{2\pi} I_{h}(|y_{h}|^{2}) (\tau_{h} \cdot \phi_{hu} ) \\
+ \lambda \int_{0}^{2\pi} x_{hu} \cdot \phi_{hu} &=0 \notag
\\\label{2.10}
\int_{0}^{2\pi} I_{h}(y_{h} \cdot \psi_{h}) |x_{hu}| + \int_{0}^{2\pi} (\tau_{h} \cdot \psi_{hu} ) &=0\\\label{2.11}
x_{h}(0, \cdot) &= I_{h} x_{0}
\end{align}
for all $\phi_{h}, \psi_{h} \in X_{h}^{n}$ and $t \in [0,T]$. Here $P_{h}$ denotes the discrete normal projection matrix $P_{h} =I - \tau_{h} \otimes \tau_{h}$ with $\tau_{h}=\frac{x_{hu}}{|x_{hu}|}$.

\begin{rem}\label{rem:energy}
The energy estimates in the discrete setting are derived in exactly the same way as in the continuous setting.
More precisely:
\begin{itemize}
\item{\rm (Step 1)} take $\phi_{h}=x_{ht} $ in \eqref{2.9}; 
\item{\rm (Step 2)} take $\psi_{h}=y_{h}$ in $( \ref{2.10})_{t}$ that is \eqref{2.10} differentiated with respect to time;
\item{\rm (Step 3)} combine the two equations to obtain
\begin{align}\label{enestimate}
\int_{0}^{2\pi} I_{h} (|x_{ht}|^{2} ) |x_{hu}| + \frac{d}{dt} \left \{ 
\frac{1}{2} \int_{0}^{2\pi} I_{h} (|y_{h}|^{2}) |x_{hu}| + \frac{\lambda}{2} \int_{0}^{2\pi} |x_{hu}|^{2}
\right \}=0. 
\end{align}
\end{itemize}
The reason for highlighting the structure is that the error analysis follows essentially the same steps.
\end{rem}

\subsection{Statement of the main results}

As pointed out in \cite{DD09} the initial datum \eqref{2.11} determines the curvature vector $y_{h}(0, \cdot) \in X_{h}^{n}$ through \eqref{2.10}, that is through the relation
\begin{align}\label{2.12}
\int_{0}^{2\pi} I_{h} (y_{h}(0, \cdot) \cdot \psi_{h} ) |(I_{h} x_{0})_{u}| + \int_{0}^{2\pi} \frac{(I_{h} x_{0})_{u} \cdot \psi_{hu}}{|(I_{h} x_{0})_{u}|} = 0
\end{align}
for all $\psi_{h} \in X_{h}^{n}$. Furthermore the following approximation result holds:
\begin{lemma}\label{lem2.2}
For $h \leq h_{0}$, with $h_0$ sufficiently small, we have that 
$$ \| y (0, \cdot)-y_{h}(0, \cdot)\|_{L^{2}} \leq Ch.$$
\end{lemma}
\begin{proof}
See \cite[Lemma 2.2]{DD09}. The claim follows from \eqref{2.12} with $\psi_{h}=I_{h} y(0, \cdot) - y_{h}(0, \cdot)$ and using \eqref{2.7}, $y(0, \cdot)= \frac{1}{|x_{0u}|}\left( \frac{x_{0u}}{|x_{0u}|} \right)_{u}$ and appropriate interpolation estimates. 
\end{proof}

The main error estimates are summarized in the following theorem.
\begin{teo}\label{mainteo}
Let Assumption~\ref{assum} holds.
There exists $h_{0}>0$ such that \eqref{2.9}-\eqref{2.11} has a unique solution and 
\begin{align}\label{3.13}
&\sup_{t \in [0,T]} \| x(t, \cdot) - x_{h} (t, \cdot) \|^{2}_{H^{1}} + \int_{0}^{T} \| x_{t} (t, \cdot) - x_{ht} (t, \cdot) \|^{2}_{L^{2}} dt \leq Ch^{2},\\ \label{3.14}
&\sup_{t \in [0,T]} \| y(t, \cdot) - y_{h} (t, \cdot) \|^{2}_{L^{2}} + \int_{0}^{T} \| y_{u} (t, \cdot) - y_{hu} (t, \cdot) \|^{2}_{L^{2}} dt \leq Ch^{2}
\end{align}
for all $0 < h \leq h_{0}$. The constant $C$ depends on $T$, $c_{0}$, $C_{0}$, $\lambda$, and on higher norms of the solution $x$ of the continuous problem. 
\end{teo}

The proof follows exactly the same strategy of the analogous problem studied in \cite[Thm 2.3]{DD09} in that we exploit the variation structure of the problem as pointed out in Remark~\ref{rem:energy}. 
However, by our choice of energy functional (that is $\mathcal{D}_{\lambda}$ over $\mathcal{E}_{\lambda}$) the main difficulty posed by the treatment of the length element (see \cite[pp 657-661]{DD09}) falls completely away, and this simplifies the analysis in a significant manner.
In fact, the simplified analysis together with the good grid properties of the discrete flow constitute the main advantages of considering \eqref{1.2} over \eqref{flowDKS}.

\section{Error analysis}
\label{sec3}

Recall that Assumption~\ref{assum} holds throughout this section.
\subsection{Preliminary results}

We start by deriving a series of estimates that will enable us to prove Theorem~\ref{mainteo}.
To that end let us assume that for $h>0$ there exists a unique solution of \eqref{2.9}-\eqref{2.11} on $[0, \bar{t}] \times [0, 2\pi] $
for some $\bar{t} \in (0,T]$. The first lemma is the analogous of (Step 1) in the derivation of the energy estimate (cf.~Remark~\ref{rem:energy}). 

\begin{lemma}\label{lem3.1} Assume the uniform bounds
\begin{align}\label{3.3a}
 \frac{c_{0}}{2} \leq |x_{hu}| \leq 2C_{0} 
\end{align}
and
\begin{align}\label{3.3b}
 |y_{h}| \leq 2 C_{0}
\end{align}
hold on $[0, \bar{t}]$.
Then, for any $t \in (0, \bar{t})$, we have
\begin{align*}
\frac{c_{0}}{4}\|x_{t}-x_{ht}\|^{2}_{L^{2}} &+\lambda \frac{d}{dt} \left(\frac{1}{2}\int_{0}^{2\pi}|x_{u}-x_{hu}|^{2} \right) \\
& \quad - \int_{0}^{2\pi} \left(\frac{1}{|x_{u}|}Py_{u} - \frac{1}{|x_{hu}|}P_{h}y_{hu}\right) \cdot ( x_{tu} -x_{htu})\\
& \quad -\frac{1}{2} \int_{0}^{2\pi} (|y|^{2}\tau -|y_{h}|^{2} \tau_{h}) \cdot( x_{tu} -x_{htu})
\\
& \leq C h^{2} + C \| x_{u} - x_{hu} \|_{L^{2}}^{2} + C \| y - y_{h} \|_{H^{1}}^{2}. 
\end{align*}
\end{lemma}
\begin{proof}
Inserting $\phi_{h} = I_{h} x_{t} - x_{ht}$ in \eqref{2.2} and \eqref{2.9}, taking the difference and using \eqref{2.7}
we obtain
\begin{align*}
&\int_{0}^{2\pi}|x_{t}-x_{ht}|^{2} |x_{hu}| +\lambda \frac{d}{dt} \left(\frac{1}{2}\int_{0}^{2\pi}|x_{u}-x_{hu}|^{2} \right)\\
& \quad - \int_{0}^{2\pi} \left(\frac{1}{|x_{u}|}Py_{u} - \frac{1}{|x_{hu}|}P_{h}y_{hu}\right) \cdot (I_{h} x_{t} -x_{ht})_{u}\\
& \quad -\frac{1}{2} \int_{0}^{2\pi} (|y|^{2}\tau -|y_{h}|^{2} \tau_{h}) \cdot(I_{h} x_{t} -x_{ht})_{u}
\displaybreak[0] \\
&=\int_{0}^{2\pi} x_{t} \cdot (I_{h} x_{t}- x_{ht}) (|x_{hu}|-|x_{u}|) +
\int_{0}^{2\pi} (x_{t}-x_{ht}) \cdot (x_{t} - I_{h} x_{t}) |x_{hu}| \\
&\quad + \sum_{j=1}^{N}\frac{h_{j}^{2}}{6} \int_{I_{j}} x_{htu} \cdot (I_{h}x_{t} -x_{ht})_{u} |x_{hu}|
 + \lambda \int_{0}^{2\pi} (x_{u}-x_{hu}) \cdot (x_{t}-I_{h} x_{t})_{u}\\
 & \quad - \sum_{j=1}^{N}\frac{h_{j}^{2}}{6} \int_{I_{j}}\frac{1}{2}|y_{hu}|^{2} \tau_{h}\cdot (I_{h} x_{t} -x_{ht})_{u}
 \quad = \sum_{i=1}^{5}S_{i}.
\end{align*}
Using \eqref{3.3a}, \eqref{2.5}, Young's inequality, and the smoothness of $x$ yields 
\begin{align*}
S_{1}+S_{2}+S_{4} \leq Ch^{2}+ \frac{1}{2}\int_{0}^{2\pi}|x_{t}-x_{ht}|^{2}|x_{hu}| + C\int_{0}^{2\pi}|x_{u}-x_{hu}|^{2}.
\end{align*}
The term $S_{3}+S_{5}$ is estimated as in \cite{DD09} using Young inequality several times with some $\delta > 0$ small enough, \eqref{3.3a}, \eqref{nostra}, \eqref{invest} \eqref{3.3b}
and smoothness assumptions on $x$, 
namely 
\begin{align*}
S_{3}+S_{5} & = \sum_{j=1}^{N}\frac{h_{j}^{2}}{6}\int_{I_{j}} \{ -|(I_{h}x_{t} -x_{ht})_{u}|^{2} + (I_{h}x_{t})_{u} \cdot (I_{h}x_{t} -x_{ht})_{u} \}|x_{hu}| \\
& \quad - \sum_{j=1}^{N}\frac{h_{j}^{2}}{6} \int_{I_{j}}\frac{1}{2} |y_{hu}| \, |(y_{hu} - y_{u}) + y_{u}| \, \tau_{h}\cdot (I_{h} x_{t} -x_{ht})_{u}
\displaybreak[0] \\
& \leq -\sum_{j=1}^{N}\frac{h_{j}^{2}}{6}\int_{I_{j}} \frac{1}{2}|(I_{h}x_{t} -x_{ht})_{u}|^{2} \}|x_{hu}| + \sum_{j=1}^{N} \frac{h_{j}^{2}}{6} \int_{I_{j}} \frac{1}{2} |(I_{h}x_{t})_{u}|^{2} |x_{hu}|\\
& \quad + \sum_{j=1}^{N}\frac{h_{j}^{2}}{6} \int_{I_{j}} \frac{1}{2} |y_{hu}| \, |y_{hu} - y_{u}| \, |(I_{h} x_{t} -x_{ht})_{u}| \\
& \quad + \sum_{j=1}^{N}\frac{h_{j}^{2}}{6} \int_{I_{j}} \frac{1}{2} |y_{hu} - y_{u}| \, |y_{u}| \, |(I_{h} x_{t} -x_{ht})_{u}| \\
& \quad + \sum_{j=1}^{N}\frac{h_{j}^{2}}{6} \int_{I_{j}} \frac{1}{2} |y_{u}|^{2} \, |(I_{h} x_{t} -x_{ht})_{u}| \displaybreak[0] \\
& \leq -\sum_{j=1}^{N}\frac{h_{j}^{2}}{6}\int_{I_{j}} \frac{1}{2}|(I_{h}x_{t} -x_{ht})_{u}|^{2} |x_{hu}| + C h^{2} \| x_{tu}\|_{L^{2}}^{2} \\
& \quad + C_\delta \sum_{j=1}^{N}\frac{h_{j}^{2}}{6} (\| y_{hu} \|_{L^{\infty}}^{2} + \| y_{u} \|_{L^{\infty}}^{2})\int_{I_{j}}|y_{hu} - y_{u}|^{2} 
 + C_\delta h^{2} \| y_{u}\|_{L^{2}}^{2} \| y_{u}\|_{L^{\infty}}^{2} \\ 
& \quad + \delta \sum_{j=1}^{N} \frac{h_{j}^{2}}{6} \int_{I_{j}} |(I_{h}x_{t} - x_{ht})_{u}|^{2} |x_{hu}| \\
% & \leq C h^{2} \| x_{tu}\|_{L^{2}}^{2} + C \sum_{j=1}^{N}\frac{h_{j}^{2}}{6} (\| y_{hu} \|_{L^{\infty}}^{2} + \| y_{u} \|_{L^{\infty}}^{2})\int_{I_{j}}|y_{hu} -y_{u}|^{2} + C h^{2} \| y_{u}\|_{L^{2}}^{2} \| y_{u}\|_{L^{\infty}}^{2} \\
& \leq C h^{2} + C (\|y_{h}\|_{L^{\infty}}^{2} +h^{2}) \| y_{hu} -y_{u}\|_{L^{2}}^{2} \leq Ch^{2} + C \| y_{hu} -y_{u}\|_{L^{2}}^{2}. 
\end{align*}

The next estimate is provided in \cite[Lemma~3.1]{DD09} to which we refer for details: it is derived with help of \eqref{2.5} and \eqref{2.8} and the fact that $\int_{0}^{2\pi} x_{tu} \cdot z_{h} = \int_{0}^{2 \pi} (I_{h}x_{t})_{u} \cdot z_{h} $ holds for any $z_{h} \in Z_{h}$. We have that 
\begin{align*}
\int_{0}^{2\pi} \left(\frac{1}{|x_{u}|}Py_{u} - \frac{1}{|x_{hu}|}P_{h}y_{hu}\right) \cdot (I_{h} x_{t})_{u}
\leq Ch^{2}+ \int_{0}^{2\pi} \left(\frac{1}{|x_{u}|}Py_{u} - \frac{1}{|x_{hu}|}P_{h}y_{hu}\right) \cdot x_{tu}.
\end{align*}
Finally, we can show with the help of \eqref{3.1} and \eqref{3.3b} that 
\begin{align*}
& \frac{1}{2} \int_{0}^{2\pi} (|y|^{2}\tau -|y_{h}|^{2} \tau_{h}) \cdot(I_{h} x_{t})_{u} \\
& = \frac{1}{2} \int_{0}^{2\pi} (|y|^{2}\tau -|y_{h}|^{2} \tau_{h}) \cdot ((I_{h} x_{t})_{u} -x_{tu}) + \frac{1}{2} \int_{0}^{2\pi} (|y|^{2}\tau -|y_{h}|^{2} \tau_{h}) \cdot x_{tu} \\
&\leq Ch^{2}+ C \| y-y_{h}\|_{L^{2}}^{2}+ C\|\tau-\tau_{h}\|_{L^{2}}^{2} + \int_{0}^{2\pi} (|y|^{2}\tau -|y_{h}|^{2} \tau_{h}) \cdot x_{tu}.
\end{align*}
Putting all estimates together, using $|x_{hu}| \geq c_{0}/2$ and \eqref{pipe} the claim follows.
\end{proof}

In the second lemma we emulate (Step 2) in the derivation of the energy estimate (cf.~Remark~\ref{rem:energy}).
\begin{lemma}\label{lem3.2}
Assume that \eqref{3.3a} and \eqref{3.3b} hold on $ [0, \bar{t}]$. Then for any $\epsilon >0$ and any $t \in (0, \bar{t})$ we have that 
\begin{align*}
\frac{1}{2} &\frac{d}{dt} \int_{0}^{2\pi} I_{h} (|I_{h}y -y_{h}|^{2}) |x_{hu}|
+\int_{0}^{2\pi} \left( \frac{1}{|x_{u}|} Px_{tu} -\frac{1}{|x_{hu}|} P_{h} x_{htu} \right) \cdot (y_{u}-y_{hu})\\
& -\frac{1}{2} \int_{0}^{2\pi} |y-y_{h}|^{2} x_{htu} \cdot \tau_{h} + \int_{0}^{2\pi} (( x_{tu} \cdot \tau) y -(x_{htu} \cdot \tau_{h}) y_{h}) \cdot (y-y_{h})\\
& \leq \epsilon \| x_{t} - x_{ht} \|_{L^{2}}^{2} + C_{\epsilon} (h^{2} + \| y-y_{h}\|_{H^{1}}^{2} + 
\| x_{u}-x_{hu} \|^{2}_{L^{2}}
).
\end{align*}
\end{lemma}
\begin{proof}
In \cite[Lemma~3.2]{DD09} it is shown that
\begin{align*}
\frac{1}{2} &\frac{d}{dt} \int_{0}^{2\pi} I_{h} (|I_{h}y -y_{h}|^{2}) |x_{hu}|
+\int_{0}^{2\pi} \left( \frac{1}{|x_{u}|} Px_{tu} -\frac{1}{|x_{hu}|} P_{h} x_{htu} \right) \cdot (y_{u}-y_{hu})\\
& -\frac{1}{2} \int_{0}^{2\pi} |y-y_{h}|^{2} x_{htu} \cdot \tau_{h} + \int_{0}^{2\pi} (( x_{tu} \cdot \tau) y -(x_{htu} \cdot \tau_{h}) y_{h}) \cdot (y-y_{h})\\
& \leq \epsilon \| x_{t} - x_{ht} \|_{L^{2}} + C_{\epsilon} (h^{2} + \| y-y_{h}\|_{H^{1}}^{2} + \| \tau -\tau_{h} \|^{2}_{L^{2}} +
\| |x_{u}| - |x_{hu}| \|^{2}_{L^{2}}
)
\end{align*}
holds for any $ \epsilon >0$. The above estimate is obtained as follows: equations \eqref{2.3} and \eqref{2.10} are differentiated with respect to time and subtracted. The estimate is achieved by choosing $\psi_{h}=I_{h} y - y_{h}$, manipulating the equation with help of \eqref{2.7}, and then estimating using \eqref{2.5}, inverse estimates, the assumptions \eqref{3.3a} and \eqref{3.3b} and the regularity of the smooth solution. 

On the other hand the uniform control from above and below of the lengths elements $|x_{u}|$ and $|x_{hu}|$ yields
\begin{align}\label{pipe}
|\tau- \tau_{h}| + \big| |x_{u}| - |x_{hu}| \big| \leq C |x_{u}-x_{hu}|
\end{align}
and the claim follows.
\end{proof}

Next we combine the previous two lemmas: unlike (Step 3) in the derivation of the energy estimate (cf.~Remark~\ref{rem:energy}), several terms involving the time derivative of the length elements and tangents do not cancel out.
These are collected in the term $\zeta$ below and must be taken care of (as explained in Lemma~\ref{lem3.4}) when we apply a Gronwall argument later on.

\begin{lemma}\label{lem3.3}
Assume \eqref{3.3a} and \eqref{3.3b} hold on $[0, \bar{t}]$. Then for ant $t \in (0, \bar{t})$ we have
\begin{align*}
\frac{c_{0}}{16} \|x_{t}-x_{ht}\|^{2}_{L^{2}} +\zeta'(t) & \leq Ch^{2}+
C(\| y-y_{h}\|_{H^{1}}^{2} + \| x_{u} -x_{hu}\|_{L^{2}}^{2}) 
\end{align*}
where 
\begin{align*}
\zeta(t)&=\frac{\lambda}{2} \| x_{u} -x_{hu}\|_{L^{2}}^{2} + \frac{1}{2} \int_{0}^{2\pi} I_{h} (|I_{h}y -y_{h}|^{2}) |x_{hu}|
-\frac{1}{4} \int_{0}^{2\pi }|y|^{2} |\tau_{h} -\tau|^{2} |x_{hu}|\\
&+ \int_{0}^{2\pi} \left( \frac{|x_{hu}| -|x_{u}|}{|x_{u}| } (\tau_{h} -\tau) 
+\frac{1}{2} \frac{|x_{hu}|}{|x_{u}|} |\tau-\tau_{h}|^{2} \tau \right) \cdot y_{u}.
\end{align*}
\end{lemma}
\begin{proof}
If follows from Lemma~\ref{lem3.1} and \ref{lem3.2} with $\epsilon =c_{0}/8$ that
\begin{multline*}
\frac{c_{0}}{8}\|x_{t}-x_{ht}\|^{2}_{L^{2}} + \lambda \frac{d}{dt} \left(\frac{1}{2}\int_{0}^{2\pi}|x_{u}-x_{hu}|^{2} \right) 
+ \frac{1}{2} \frac{d}{dt} \int_{0}^{2\pi} I_{h} (|I_{h}y -y_{h}|^{2}) |x_{hu}|
\\ + \int_{0}^{2\pi} (A + B)
 \leq C h^{2} + C\|x_{u}-x_{hu}\|^{2}_{L^{2}} + C \| y - y_{h} \|_{H^{1}}^{2}
\end{multline*}
where
\begin{align*}
A&= -\frac{1}{2}(|y|^{2}\tau -|y_{h}|^{2} \tau_{h}) \cdot( x_{tu} -x_{htu})
-\frac{1}{2}|y-y_{h}|^{2} (x_{htu}, \tau_{h}) \\
&\qquad + (( x_{tu} \cdot \tau) y -(x_{htu} \cdot \tau_{h}) y_{h}) \cdot (y-y_{h}),
\\
B &= 
\left( \frac{1}{|x_{u}|} Px_{tu} -\frac{1}{|x_{hu}|} P_{h} x_{htu} \right) \cdot (y_{u}-y_{hu})\\
&\qquad - \left(\frac{1}{|x_{u}|}Py_{u} - \frac{1}{|x_{hu}|}P_{h}y_{hu}\right) \cdot ( x_{tu} -x_{htu}).
\end{align*}
Notice that for instance $\frac{1}{|x_{u}|} Px_{tu} -\frac{1}{|x_{hu}|} P_{h} x_{htu} = (\tau- \tau_{h})_{t}$. Similar considerations for the other terms appearing in $A+B$ indicate that we can not estimate these terms yet.
In \cite[Lemma~3.3]{DD09} the authors show by a smart manipulation of the terms that $A$ can be written as follows:
\begin{align*}
&A=-\frac{1}{4}|y|^{2} \frac{d}{dt} \Big{(} |\tau_{h} -\tau|^{2} |x_{hu}| \Big{)}
+ \frac{1}{2} |y|^{2} \Big{(} 1 - \frac{|x_{hu}|}{|x_{u}|} \Big{)} (\tau_{h} -\tau) \cdot x_{tu} \\
& \quad - \frac{1}{4} |y|^{2} \frac{|x_{hu}|}{|x_{u}|} (\tau \cdot x_{tu}) |\tau_{h} - \tau|^{2} +\frac{1}{2} |y-y_{h}|^{2} (x_{tu} \cdot \tau_{h})
+ (y \cdot(y-y_{h}) ) (x_{tu} \cdot (\tau -\tau_{h})).
\end{align*}
Using \eqref{3.3a}, \eqref{3.1}, \eqref{pipe}, and the smoothness assumptions on $x$ and $y$ 
we obtain that
\begin{align*}
\int_{0}^{2\pi} A \,\, &\geq %-\frac{1}{4} \frac{d}{d t} \int_{0}^{2\pi }|y|^{2} |\tau_{h} -\tau|^{2} |x_{hu}|\\
%& 
- C(\| y-y_{h}\|_{L^{2}}^{2} + \| x_{u} -x_{hu}\|_{L^{2}}^{2}) -\frac{1}{4} \frac{d}{d t} \int_{0}^{2\pi }|y|^{2} |\tau_{h} -\tau|^{2} |x_{hu}|.
\end{align*}
Next, it is shown in \cite[$(3.13)$]{DD09} that
\begin{align*}
B = y_{u} \cdot \frac{d}{dt} \left( \frac{1}{|x_{u}|} P x_{hu} -\tau_{h} + \tau \right) + x_{tu} \cdot z_{1} + x_{tu} \cdot z_{2}
\end{align*}
where
\begin{align*}
z_{1}=\frac{1}{|x_{hu}|} \left( \frac{|x_{hu}|}{|x_{u}|}-1\right)^{2 } Py_{u} + (\tau \cdot \tau_{h} -1) 
\frac{|x_{hu}|}{|x_{u}|^{2}} Py_{u} + \left( \frac{1}{|x_{u}|} -\frac{1}{|x_{hu}|} \right)P(y_{u} -y_{hu})
\end{align*}
and
\begin{align*}
z_{2}&=y_{u}\cdot(\tau_{h} -\tau) \left( \frac{|x_{hu}|}{|x_{u}|^{2} }\tau - \frac{1}{|x_{hu}|} \tau_{h} \right)
+ (y_{u} \cdot \tau) \left( \frac{|x_{hu}|}{|x_{u}|^{2} } -\frac{1}{|x_{hu}|} \right) (\tau_{h} -\tau )\\
& \quad + 2 \frac{|x_{hu}|}{|x_{u}|^{2} } (\tau \cdot y_{u}) (1 - \tau \cdot \tau_{h}) \tau +
\frac{1}{|x_{hu}|} \left( \, ( (y_{hu}-y_{u}) \cdot \tau) \tau - ((y_{hu}-y_{u}) \cdot \tau_{h}) \tau_{h} \, \right).
\end{align*}
Observing that
\begin{align*}
\frac{1}{|x_{u}|} P x_{hu} -\tau_{h} +\tau = \left( \frac{|x_{hu}|}{|x_{u}| } -1 \right) (\tau_{h} -\tau) 
+\frac{1}{2} \frac{|x_{hu}|}{|x_{u}|} |\tau-\tau_{h}|^{2} \tau
 \end{align*}
 as well as
 \begin{align*}
 (1-\tau\cdot \tau_{h}) =\frac{1}{2}|\tau-\tau_{h}|^{2} \qquad \text{ and } \qquad 
 \frac{|x_{hu}|}{|x_{u}|^{2} }-\frac{1}{|x_{hu}|}= \frac{ (|x_{hu}| - |x_{u}|) ( |x_{hu}| + |x_{u}|)}{ |x_{u}|^{2} |x_{hu}|}
 \end{align*}
 and using \eqref{3.3a}, \eqref{3.1}, and the regularity assumptions 
 we obtain
 \begin{align*}
 \int_{0}^{2\pi} B \,\, &\geq - C(\| y_{u}-y_{hu}\|_{L^{2}}^{2} + \| x_{u} -x_{hu}\|_{L^{2}}^{2}) \\
 & \quad +
 \frac{d}{d t} \left(\int_{0}^{2\pi } \left( \left( \frac{|x_{hu}|}{|x_{u}| } -1 \right) (\tau_{h} -\tau) 
+\frac{1}{2} \frac{|x_{hu}|}{|x_{u}|} |\tau-\tau_{h}|^{2} \tau \right) \cdot y_{u} \right).
 \end{align*}
 Putting together the estimates yields the result.
\end{proof}

The next lemma serves to estimate the additional errors defined in the previous lemma and collected in $\zeta$. 

\begin{lemma}\label{lem3.4}
Assume \eqref{3.3a} holds on $[0, \bar{t}]$. Then for $\epsilon >0$ and $t \in [0, \bar{t}]$ we have
\begin{align*}
\| (\tau -\tau_{h})(t) \|_{L^{2}}^{2} \leq \epsilon \big{(} \| (y-y_{h})(t) \|_{L^{2}}^{2} + \| (x_{u} - x_{hu})(t)\|_{L^{2}}^{2} \big{)} + C_{\epsilon} \big{(} h^{2} + \| (x -x_{h})(t) \|_{L^{2}}^{2} \big{)}.
\end{align*}
In particular it follows for $\zeta(t)$ defined as in Lemma~\ref{lem3.3} that
\begin{align*}
\zeta(t) \geq \frac{c_{0}}{16} \| (y- y_{h})(t) \|_{L^{2}}^{2} +\frac{\lambda}{8} \|( x_{u} -x_{hu})(t)\|_{L^{2}}^{2} - Ch^{2} 
-C \| (x -x_{h})(t) \|_{L^{2}}^{2}.
\end{align*}
\end{lemma} 
\begin{proof}
The first estimate concerning the difference in the unit tangents follows from \cite[Lemma $3.4$]{DD09}.
The idea there is to subtract \eqref{2.3} and \eqref{2.10} (which put into relation tangent and curvature vectors), take $\psi_{h}= I_{h}x- x_{h}$, use the fact that $(\tau-\tau_{h}) \cdot (x_{u}-x_{hu}) =\frac{1}{2}|\tau-\tau_{h}|^{2} (|x_{u}|+ |x_{hu}|) \geq C |\tau-\tau_{h}|^{2}$, and then estimate all appearing terms appropriately.

For the second estimate we use \eqref{2.6}, \eqref{3.3a}, \eqref{3.1}, Young's inequality, interpolation estimates, and smoothness assumptions 
to obtain 
\begin{align*}
\zeta(t) &\geq \frac{c_{0}}{8} \| (y- y_{h})(t) \|_{L^{2}}^{2} +\frac{\lambda}{4} \|( x_{u} -x_{hu})(t)\|_{L^{2}}^{2} -Ch^{2} -C \| (\tau -\tau_{h})(t) \|_{L^{2}}^{2}\\
& \geq \frac{c_{0}}{8} \| (y- y_{h})(t) \|_{L^{2}}^{2} +\frac{\lambda}{4} \|( x_{u} -x_{hu})(t)\|_{L^{2}}^{2} -Ch^{2} \\
& \qquad -
\epsilon C \big{(} \|( y-y_{h})(t) \|_{L^{2}} + \|( x_{u} - x_{hu})(t)\|_{L^{2}} \big{)} - C_{\epsilon} \big{(} h^{2} +\| (x -x_{h})(t) \|_{L^{2}}^{2} \big{)}
\end{align*}
where we have used the first statement in the second inequality. Choosing $\epsilon$ sufficiently small yields the claim.
\end{proof}

Finally, we need to express the error $\| y_{u} -y_{hu}\|_{L^{2}}$ appearing on the right-hand side on Lemma~\ref{lem3.3} in terms of suitable norms so that we can apply a Gronwall argument. 
\begin{lemma}\label{lem3.5}
Assume \eqref{3.3a} and \eqref{3.3b} hold on $[0, \bar{t}]$.
 Then for any $t \in (0, \bar{t})$ and $\epsilon >0$ we have that
\begin{align*}
\|y_{u}-y_{hu} \|_{L^{2}}^{2} \leq \epsilon \| x_{t} -x_{ht} \|_{L^{2}}^{2} + C_{\epsilon} ( h^{2} + \| x- x_{h} \|_{L^{2}}^{2} +
\| y- y_{h} \|_{L^{2}}^{2} + \| x_{u}- x_{hu} \|_{L^{2}}^{2} ).
\end{align*}
\end{lemma} 
\begin{proof}
The proof is an adaptation of \cite[Lemma~3.5]{DD09} to the present setting. 
For the convenience of the reader, we give here the main arguments and refer to \cite{DD09} wherever possible. 
We write 
\begin{align} \label{estim_lem4.5_start}
y_{u}-y_{hu} = P_{h} (y_{u}-y_{hu}) + \big{(} (y_{u}-y_{hu}) \cdot \tau_{h} \big{)} \tau_{h}
\end{align}
and estimate each component $\|P_{h} (y_{u}-y_{hu}) \|_{L^{2}}$ and $\|(y_{u}-y_{hu}) \cdot \tau_{h}\|_{L^{2}} $ as follows.
From \eqref{2.2} and \eqref{2.9} we infer
\begin{align}
\int_{0}^{2\pi} \frac{1}{|x_{hu}|} P_{h}(y_{u}-y_{hu}) \cdot \phi_{hu} &= \int_{0}^{2\pi} \left( \frac{1}{|x_{hu}|} P_{h}y_{u} -\frac{1}{|x_{u}|} Py_{u} \right) \cdot \phi_{hu} \nonumber \\
& \qquad + \int_{0}^{2\pi} \{ x_{t} \cdot \phi_{h} |x_{u}| - I_{h} (x_{ht} \cdot \phi_{h}) |x_{hu}| \} \nonumber \\
& \qquad -\frac{1}{2} \int_{0}^{2\pi} \{ |y|^{2} (\phi_{hu} \cdot \tau) - I_{h} (|y_{h}|^{2}) (\phi_{hu} \cdot \tau_{h}) \} \nonumber \\
& \qquad +\lambda \int_{0}^{2\pi} (x_{u} -x_{hu}) \cdot \phi_{hu} \qquad = \sum_{i=1}^{4} S_{i} \label{to_estim_lem3.5}
\end{align}
for $\phi_{h} \in X_{h}^{n}$. We choose $\phi_{h} =I_{h} y -y_{h}$ and estimates all terms separately.
Using that $P_{h} P_{h} = P_{h}$, the symmetry of the matrix $P_{h}$, and interpolation estimates we obtain that 
\begin{multline*}
\int_{0}^{2\pi} \frac{1}{|x_{hu}|} P_{h}(y_{u}-y_{hu}) \cdot (I_{h} y -y_{h})_{u} \\
\geq \frac{1}{2 C_{0}} \| P_{h} (y_{u}-y_{hu})\|_{L^{2}}^{2} - C_{\epsilon}h^{2} - \epsilon \| P_{h} (y_{u}-y_{hu})\|_{L^{2}}^{2} \\
\geq \frac{1}{4 C_{0}} \| P_{h} (y_{u}-y_{hu})\|_{L^{2}}^{2} - Ch^{2}.
\end{multline*}
The integral $S_{1}$ and $S_{4}$ are estimated in a standard fashion (using Young's inequality, interpolation estimates, \eqref{pipe}, and the regularity assumptions on $y$) 
to get that 
\begin{align*}
S_{1} + S_{4} \leq C_{\delta} (\| x_{u} -x_{hu} \|_{L^{2}}^{2} + \| \tau-\tau_{h}\|_{L^{2}}^{2})+ \delta \|y_{u} -y_{hu}\|_{L^{2}}^{2} + Ch^{2}.
\end{align*}
Next, using \eqref{2.7}, \eqref{2.2}, and \eqref{invest} we obtain that 
\begin{align*}
S_{2} &= \int_{0}^{2\pi} (x_{t} \cdot (I_{h} y -y_{h})) (|x_{u}| -|x_{hu}|) + \int_{0}^{2\pi} |x_{hu}| (x_{t}-x_{ht}) \cdot (I_{h} y -y_{h})
\\
& \quad - \sum_{j=1}^{N} \frac{1}{6}h_{j}^{2}\int_{I_{j}} |x_{hu}| (x_{htu} - (I_{h} x_{t})_{u}+ (I_{h} x_{t})_{u}) \cdot (I_{h} y -y_{h})_{u}
\\
& \leq C \| x_{u} -x_{hu} \|_{L^{2}}^{2} + C_{\epsilon }h^{2} + C_{\epsilon} \| y -y_{h} \|_{L^{2}}^{2} +
\epsilon \| x_{t} -x_{ht}\|_{L^{2}}^{2}\\
& \quad + C h\|(I_{h} x_{t})_{u}\|_{L^{2}}\| I_{h} y -y_{h}
 \|_{L^{2}}
+ C \| x_{ht}-I_{h}x_{t} \|_{L^{2}} \|I_{h} y -y_{h} \|_{L^{2}}\\
& \leq C \| x_{u} -x_{hu} \|_{L^{2}}^{2} + C_{\epsilon }h^{2} + C_{\epsilon} \| y -y_{h} \|_{L^{2}}^{2} +
2\epsilon \| x_{t} -x_{ht}\|_{L^{2}}^{2}.
\end{align*}
Next, we write, using \eqref{2.7}, \eqref{3.1}, \eqref{3.3a}, \eqref{3.3b}, \eqref{pipe}, both inverse estimates \eqref{invest}, and smoothness assumptions, 
\begin{align*}
 S_{3} &= -\frac{1}{2} \int_{0}^{2\pi} |y|^{2} (I_{h} y -y_{h})_{u} \cdot (\tau-\tau_{h}) 
 -\frac{1}{2} \int_{0}^{2\pi} (|y|^{2} -|y_{h}|^{2}) (I_{h} y -y_{h})_{u} \cdot \tau_{h} \\
& \quad + \sum_{j=1}^{N} \frac{1}{6} h_{j}^{2} \int_{I_{j}} |(y_{hu}-(I_{h}y)_{u}) + (I_{h}y)_{u}|^{2} (I_{h} y - y_{h})_{u} \cdot \tau_{h} \\
 & \leq \delta \| y_{u} -y_{hu} \|_{L^{2}}^{2} + C_{\delta} \big{(} \| x_{u} - x_{hu} \|_{L^{2}}^{2} + \| y - y_{h} \|_{L^{2}}^{2} + h^{2} \big{)} \\
 & \quad + C \| I_{h} y -y_{h}\|_{L^{\infty}} \| I_{h} y -y_{h}\|_{L^{2}}\| (I_{h} y -y_{h})_{u}\|_{L^{2}} \\
 & \quad + Ch^{2}\| (I_{h} y)_{u}\|_{L^{\infty}} \| (I_{h} y)_{u}\|_{L^{2}} \| (I_{h} y -y_{h})_{u}\|_{L^{2}}\\
 & \leq 3 \delta \| y_{u} -y_{hu} \|_{L^{2}}^{2} + C_{\delta} \big{(} \| x_{u} - x_{hu}\|_{L^{2}}^{2} +\| y -y_{h}\|_{L^{2}}^{2} + h^{2}\big{)}.
\end{align*}
Putting all above estimates together we obtain from \eqref{to_estim_lem3.5} that 
\begin{align}\label{intermezzo}
\frac{1}{4 C_{0}} \| P_{h} (y_{u}-y_{hu})\|_{L^{2}}^{2} &\leq 4 \delta \| y_{u} -y_{hu} \|_{L^{2}}^{2} + 2\epsilon \| x_{t} -x_{ht}\|_{L^{2}}^{2} \\
& \quad + C_{\delta, \epsilon} (\| x_{u} - x_{hu}\|_{L^{2}}^{2} +\| y -y_{h}\|_{L^{2}}^{2} + h^{2}). \notag
\end{align}
The estimate
\begin{align*}
\int_{0}^{2\pi} |(y_{u}-y_{hu}) \cdot \tau_{h} | \leq C (h^{2} + \| \tau- \tau_{h}\|_{L^{2}} ^{2} + \| y-y_{h}\|_{L^{2}}^{2})
+ \| |x_{u}| - |x_{hu}| \|_{L^{2}}^{2}
\end{align*}
is provided in \cite[Lemma~3.5]{DD09}, to which we refer for more details. 
Note that this estimate requires the assumption \eqref{extra-ass}, which we haven't used elsewhere yet. 
Recalling \eqref{estim_lem4.5_start}, the claim now follows from the above estimate, \eqref{pipe}, \eqref{intermezzo}, by choosing $\delta$ appropriately small. 
\end{proof}

\subsection{Proof of the main result, Theorem \ref{mainteo}}

The proof is a straight-forward adapation of the techniques presented in \cite[Theorem~2.3]{DD09} to the present setting. We give it here for the sake of the reader.
First of all notice that from standard ODE theory we have local existence and uniqueness of a discrete solution $(x_{h}, y_{h})$ of \eqref{2.9}, \eqref{2.10}, \eqref{2.11} on $[0,T_{h}]$, $T_{h}>0$, for $h \leq h_{0}$ small enough.
Upon recalling \eqref{3.1}, define
\begin{align*}
\hat{T}_{h} := \sup \{ &t \in [0,T]\, :\, (x_{h}, y_{h}) \text{ solves } \eqref{2.9}, \eqref{2.10}, \eqref{2.11} \text{ on } [0,t] \text{ and }\\
& \frac{c_{0}}{2}  \leq |x_{hu}| \leq 2 C_{0}, \quad |y_{h}| \leq 2 C_{0} \text{ in } [0,t] \times [0, 2\pi]
\}.
\end{align*}
We show next that $\hat{T}_{h}=T$ if $h \leq h_{0}$ is sufficiently small.
Integrating in time the inequality obtained in Lemma~\ref{lem3.3} and using the fact that $\zeta(0) \leq Ch^{2}$ by Lemma~\ref{lem2.2} yields for $t \in (0,\hat{T}_{h}] $ that 
\begin{align*}
\frac{c_{0}}{16} 
\int_{0}^{t}\|x_{t}-x_{ht}\|^{2}_{L^{2}} dt' +\zeta(t) & \leq Ch^{2}+
C\int_{0}^{t} \big{(} \| y-y_{h}\|_{H^{1}}^{2} + \| x_{u} -x_{hu}\|_{L^{2}}^{2} \big{)} dt'.
\end{align*}
Using the second statement in Lemma~\ref{lem3.4} gives
\begin{align*}
\frac{c_{0}}{16} 
\int_{0}^{t} &\|x_{t}-x_{ht}\|^{2}_{L^{2}} dt' + \frac{c_{0}}{16} \| (y- y_{h})(t) \|_{L^{2}}^{2} +\frac{\lambda}{8} \|( x_{u} -x_{hu})(t)\|_{L^{2}}^{2} \\ 
&\leq Ch^{2} + C\| (x -x_{h})(t) \|_{L^{2}}^{2}
 +
C\int_{0}^{t}(\| y-y_{h}\|_{H^{1}}^{2} + \| x_{u} -x_{hu}\|_{L^{2}}^{2}) dt'.
\end{align*}
Adding $ \|( x -x_{h})(t)\|_{L^{2}}^{2} $ to both sides of the inequality, using the fact that
\begin{align*}
\|(x -x_{h})(t)\|_{L^{2}}^{2} &= \|( x -x_{h})(0)\|_{L^{2}}^{2} + 2 \int_{0}^{t} \int_{0}^{2\pi} ( x -x_{h})(t') \cdot ( x_{t} -x_{ht})(t') dx dt' \\
& \leq Ch^{2} + \epsilon \int_{0}^{t} \|x_{t}-x_{ht}\|^{2}_{L^{2}} dt' + C_{\epsilon} \int_{0}^{t} \|x-x_{h}\|^{2}_{L^{2}} dt'
\end{align*}
and using Lemma~\ref{lem3.5} we obtain for appropriate choice of $\epsilon$ that 
\begin{align*}
\frac{c_{0}}{32} 
\int_{0}^{t} &\|x_{t}-x_{ht}\|^{2}_{L^{2}} dt' + \frac{c_{0}}{16} \| (y- y_{h})(t) \|_{L^{2}}^{2} +\frac{\lambda}{8} \|( x_{u} -x_{hu})(t)\|_{L^{2}}^{2} + \|( x -x_{h})(t)\|_{L^{2}}^{2} \\
&\leq Ch^{2}+
C\int_{0}^{t}(\| y-y_{h}\|_{L^{2}}^{2} + \| x_{u} -x_{hu}\|_{L^{2}}^{2} +\| x -x_{h} \|_{L^{2}}^{2} )dt'.
\end{align*}
The Gronwall Lemma implies that 
\begin{align}\label{3.39}
\sup_{[0,\hat{T}_{h}]} \big( \|( y-y_{h}) (t)\|_{L^{2}}^{2} + \| (x -x_{h})(t) \|_{H^{1}}^{2} \big) +
\int_{0}^{\hat{T}_{h}} &\|x_{t}-x_{ht}\|^{2}_{L^{2}} dt' \leq Ch^{2}.
\end{align}
Together with Lemma~\ref{lem3.5} we have therefore verified the bounds \eqref{3.13}, \eqref{3.14} on $[0, \hat{T}_{h}]$.
Note that the constants depend on $c_{0}$, $C_{0}$, $T$, $\lambda$, and norms of the continuous solution~$x$ only. 

We can now prove that $\hat{T}_{h}=T$. If we had $\hat{T}_{h} < T$, then \eqref{3.39} and an inverse estimate would imply
\begin{align*}
 \| (x_{u} -x_{hu})(t) \|_{L^{\infty}} \leq C \sqrt{h}, \qquad \|( y-y_{h}) (t)\|_{L^{\infty}} \leq C \sqrt{h}. 
\end{align*}
Combined with \eqref{3.1} this would yield that 
\begin{align*}
 \frac{3c_{0}}{4}  \leq |x_{hu}|  \leq \frac{3}{2} C_{0}, \qquad |y_{h}| \leq \frac{3}{2} C_{0} \text{ in } [0,\hat{T}_{h}] \times [0, 2\pi]
\end{align*}
provided $h \leq h_{0}$, with $h_{0}$ small enough. This would contradict the maximality of $\hat{T}_{h}$. Therefore $\hat{T}_{h}=T$, and the theorem is proved.

\section{Numerical experiments}
\label{sec:numerics}

In this section we present a couple of computational results that support the results in Theorem \ref{mainteo} and illustrate the evolution and potential benefits of the new approach. 

\subsection{Time discretization} 

We discretize \eqref{2.9}, \eqref{2.10}, \eqref{2.11} in time analogously to \cite{DD09}. Let $\delta>0$ denote the (constant) time step size and let $m_{T} \in \mathbb{N}$ the number of time steps (so that $T = m_{T} \delta$).

\begin{algo}
Set $x^{(0)} = I_{h} x_{0}$ and successively determine $x_{h}^{(m+1)}$, $y_{h}^{(m+1)} \in X_{h}^{n}$ for $m=0,1,2, \ldots,m_{T}-1$ by solving the linear problem  
\begin{align} 
\int_{0}^{2\pi}I_{h} \Big{(} \frac{x_{h}^{(m+1)} -x_{h}^{(m)}}{\delta}\cdot \phi_{h} \Big{)} |x_{hu}^{(m)}| - \int_{0}^{2\pi} \frac{P_{h}^{(m)} y_{hu}^{(m+1)} \cdot \phi_{hu}}{|x_{hu}^{(m)}|} \nonumber \\
 -\frac{1}{2}
\int_{0}^{2\pi} I_{h} \big{(} |y_{h}^{(m)}|^{2} \big{)} \Big{(} \frac{x_{hu}^{(m+1)}}{|x_{hu}^{(m)}|} \cdot \phi_{hu} \Big{)} 
+ \lambda \int_{0}^{2\pi} x_{hu}^{(m+1)} \cdot \phi_{hu} &=0, \label{2.9d} \\
\int_{0}^{2\pi} I_{h} \big{(} y_{h}^{(m+1)} \cdot \psi_{h} \big{)} |x_{hu}^{(m)}| + \int_{0}^{2\pi} \Big{(} \frac{x_{hu}^{(m+1)}}{|x_{hu}^{(m)}|} \cdot \psi_{hu} \Big{)} &=0, \label{2.10d} 
\end{align}
for all test functions $\varphi_{h}, \psi_{h} \in X_{h}^{n}$. 
Here, $P_{h}^{(m)} = Id - \tau_{h}^{(m)} \otimes \tau_{h}^{(m)}$ with $\tau_{h}^{(m)} = \frac{x_{hu}^{(m)}}{|x_{hu}^{(m)}|}$. 
\end{algo}
Writing 
$$x_{h}^{(m)} = \sum_{j=1}^N x_{j}^{(m)} \varphi_{j} \mbox{ and } y_{h}^{(m)} = \sum_{j=1}^N y_{j}^{(m)} \varphi_{j} \quad \mbox{with }x_{j}^{(m)}, y_{j}^{(m)} \in \R^{n},$$ 
the problem for time step $m$ can be formulated in terms of the coefficients and written in the form 
\begin{align*}
0 = \, & \frac{1}{2 \delta} (q_{j}^{(m)} + q_{j+1}^{(m)}) (x_{j}^{(m+1)} - x_{j}^{(m)}) \\
& + \frac{1}{q_{j+1}^{(m)}} P^{(m)}_{j+1} (y^{(m+1)}_{j+1} - y^{(m+1)}_{j})
-\frac{1}{q_{j}^{(m)}} P^{(m)}_{j} (y^{(m+1)}_{j} - y^{(m+1)}_{j-1}) \\
& + d^{(m)}_{j+1} (x_{j+1}^{(m+1)} -x_{j}^{(m+1)}) - d^{(m)}_{j} (x_{j}^{(m+1)} -x_{j-1}^{(m+1)}) 
\\
&+\lambda \Big{(} -\frac{1}{h_{j}} x_{j-1}^{(m+1)} + \big( \frac{1}{h_{j}} + \frac{1}{h_{j+1}} \big) x_{j}^{(m+1)} -\frac{1}{h_{j+1}} x_{j+1}^{(m+1)} \Big{)}, 
\\
0 = \, &\frac{1}{2} (q_{j}^{(m)} + q_{j+1}^{(m)}) y_{j}^{(m+1)} - \Big{(} \frac{1}{q_{j+1}^{(m)}} x_{j+1}^{(m+1)} - \big( \frac{1}{q_{j+1}^{(m)}}+\frac{1}{q_{j}^{(m)}} \big) x_{j}^{(m+1)} + \frac{1}{q_{j}^{(m)}}x_{j-1}^{(m+1)} \Big{)}, 
\end{align*}
where
\begin{align*}
 & q_{j}^{(m)}=|x_{j}^{(m)} -x_{j-1}^{(m)}|, \quad d_{j}^{(m)}= \frac{1}{4 q^{(m)}_{j}}(|y_{j}^{(m)}|^{2} + |y_{j-1}^{(m)}|^{2}), \mbox{ and } \\
 & P_{j}^{(m)}= Id - \tau_{j}^{(m)} \otimes \tau_{j}^{(m)} \mbox{ with } \tau_{j}^{(m)} = \frac{x_{j}^{(m)} -x_{j-1}^{(m)}}{q_{j}^{(m)}}.
\end{align*}
For our simulations we chose equi-distributed mesh points $u_{j} = j h$, $j = 0 \dots, N$, with $h = 2 \pi /N$. We used the commercial software Matlab \cite{MATLAB:2022a} to implement the scheme.

\subsection{Circular initial condition} 
\label{subsec:circ}

\begin{table}
{\footnotesize
\begin{center}
\begin{tabular}{|r|l|r|l||l|l|} \hline 
$N$ & $h$ & $m_{T}$ & $\delta$ & $\err$ & $\eoc$ \\ \hline \hline 
           20 &     0.31416 &         400 &      0.0025 &   1.556e-05 &          -- \\ \hline 
           30 &     0.20944 &         900 &   0.0011111 &  3.0805e-06 &      3.9944 \\ \hline 
           36 &     0.17453 &        1296 &   0.0007716 &  1.4864e-06 &       3.997 \\ \hline 
           46 &     0.13659 &        2116 &  0.00047259 &  5.5786e-07 &       3.998 \\ \hline 
           60 &     0.10472 &        3600 &  0.00027778 &  1.9279e-07 &      3.9988 \\ \hline 
           \end{tabular}
\end{center}
}
\caption{Errors and experimental orders of convergence for the self-similarly shrinking circle. See Section \ref{subsec:circ} for details.} 
\label{tab:conv_circ}
\end{table}

\begin{figure}
\begin{center}
 \includegraphics[width=4.2cm]{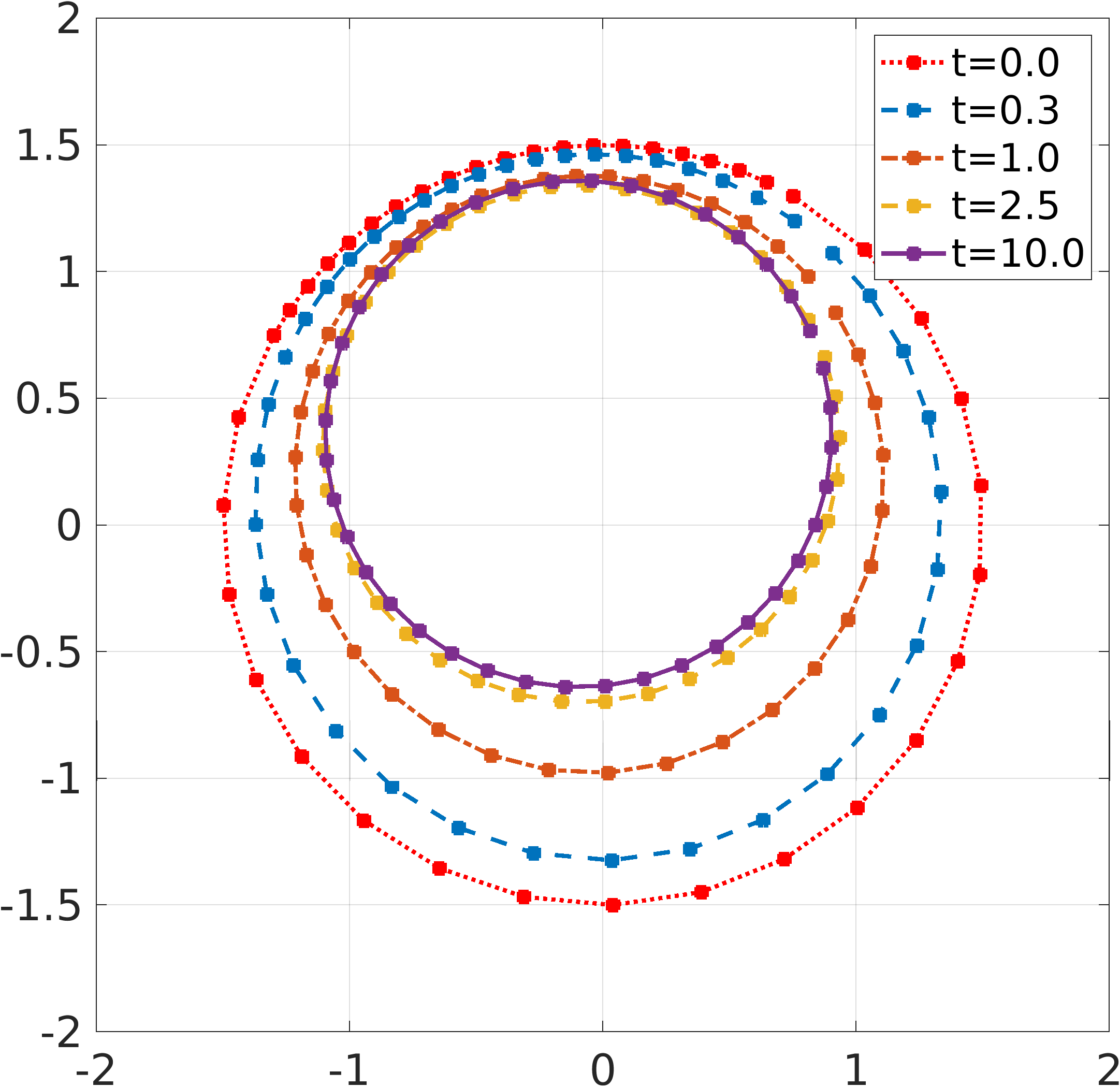} 
 \hfill 
 \includegraphics[width=4.2cm]{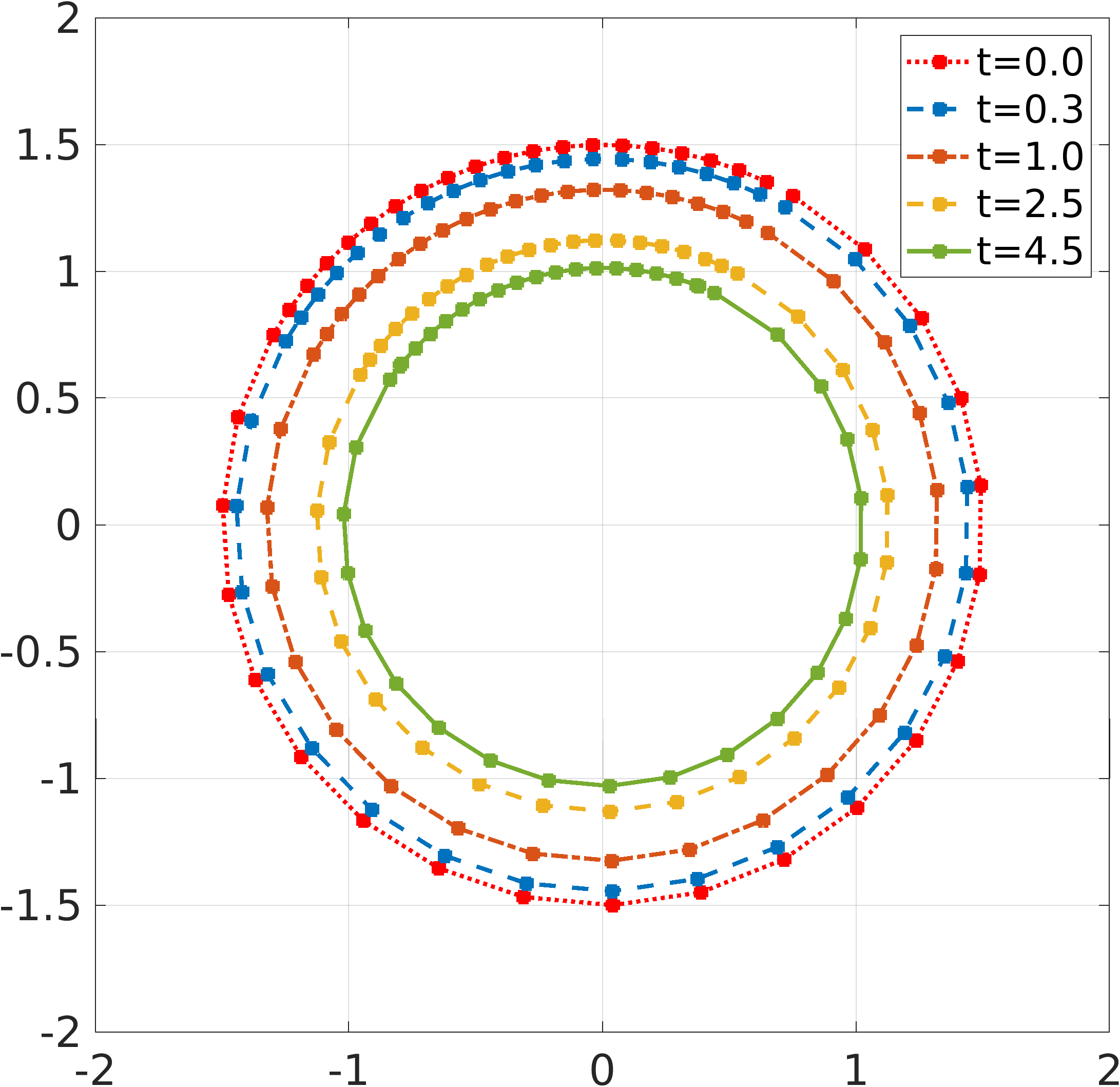} 
 \hfill 
 \includegraphics[width=4.2cm]{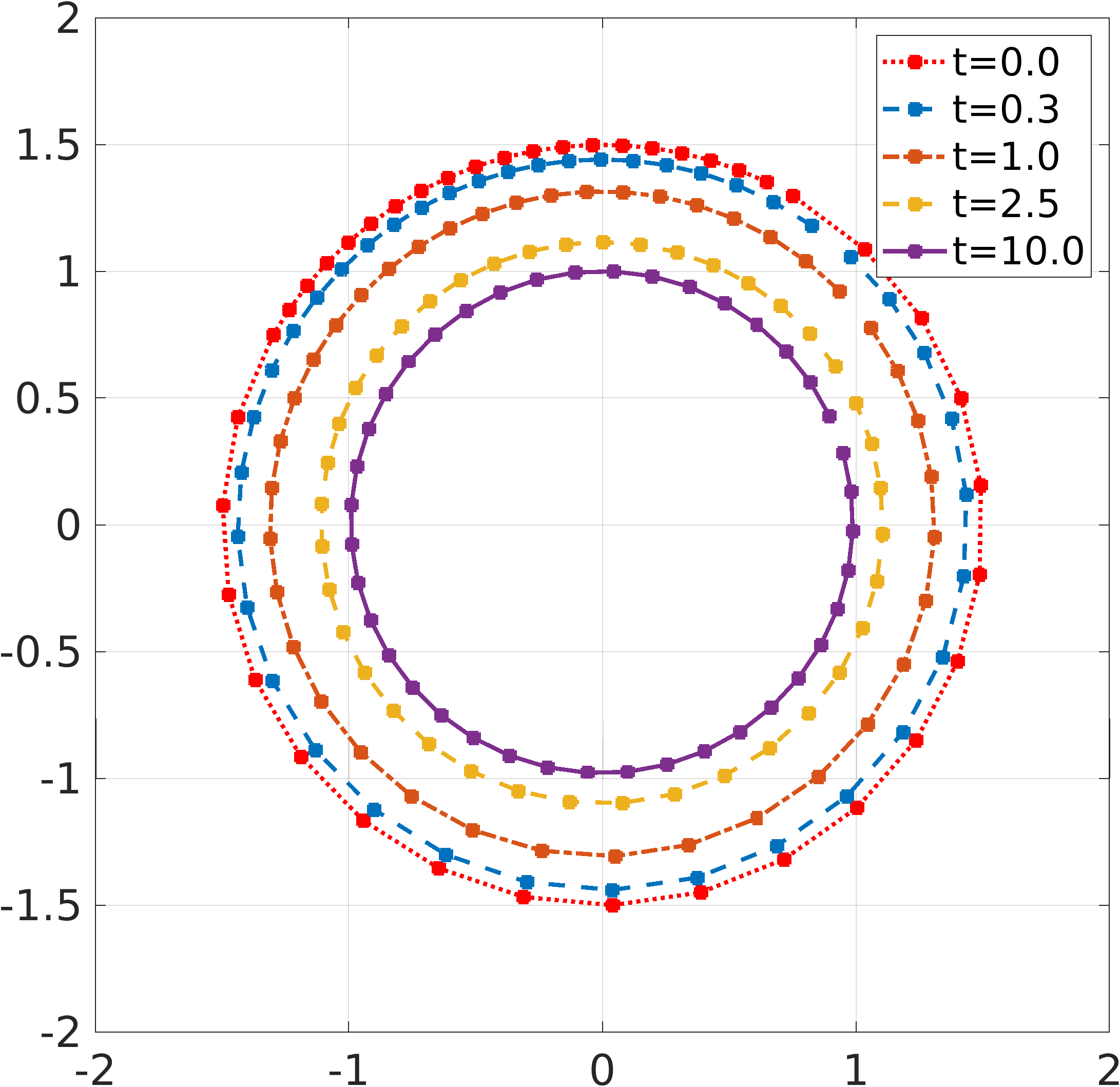} 
\end{center}
\caption{Discrete solutions at several times for the same initial condition (nonequidistributed points along a circle) but for different schemes. The new scheme \eqref{2.9d}, \eqref{2.10d} for \eqref{1.2} on the left, the scheme from \cite{DD09} for \eqref{flowDKS} in the middle, and the variant \eqref{eq:wf1dis}, \eqref{eq:wf2dis} on the right. See Sections \ref{subsec:circ} and \ref{subsec:var} for more detail. 
Note that here and in the following figures displaying solutions we plot the vertex values $x_{j}^{(m)} = x^{(m)}_{h}(u_{j})$ and connect subsequent positions but with the exception of the first and the last vertices. There is a small gap between $x_{1}^{(m)}$ and $x_{N}^{(m)}$ indicating where the parametrization starts and ends.
}
\label{fig:circ_nonequi}
\end{figure}

Recall from Example \ref{esempio} that a circle with radius $R(t)$ solves \eqref{1.2} (and thus also the weak formulation \eqref{2.2}, \eqref{2.3}) if $\dot{R}(t) = (1 - 2 \lambda R(t)^3) / (2 R(t)^3)$. We now aim for approximating such a solution and 
for supporting the theoretical result 
in Theorem \ref{mainteo}. 

We choose $\lambda = 1/2$ and an initial radius of $R(0) = 3/2$. The ODE for $R$ is solved with a standard numerical ODE solver until time $T=1$ (and still denoted with $R$). For the approximation with the scheme \eqref{2.9d}, \eqref{2.10d} we choose the initial condition $x_{0}(u) = \tfrac{3}{2} (\cos(u), \sin(u))$, which results in equi-distributed marker points $x_{j}^{(0)}$, $j=1, \dots, N$, along the circle. Below, we report on results for the time step size $\delta = 1 / m_{T}$ with $m_{T} = N^2$ for a given number of spatial mesh points $N$. These were confirmed with a smaller time step size, namely for $m_{T} = 2 N^2$ points in time. As a measure of the error, we monitored the deviation of curvature during the time stepping by computing 
\[
 \err^{(m)} = \frac{2\pi}{N} \sum_{j=1}^{N} \Big{|} y(t^{(m)},u_j) - y_{j}^{(m)} \Big{|}^2, \quad \err = \max_{1 \leq m \leq m_{T}} \err^{(m)}, 
\]
where $y(t,u) = \frac{-1}{R(t)} \frac{x_0(u)}{| x_0(u) |}$. Note that this is an approximation of the first term in the error estimate \eqref{3.14}. 

We observed that the scheme produced approximately circular shapes at each time step that were centred around the origin. The errors for several values of $N$ are stated in Table \ref{tab:conv_circ}. The experimental orders of convergence, $\eoc$, are close to four. The theoretical results \eqref{3.14} in Theorem \ref{mainteo} only predicts second order convergence, however the theoretical estimate might not be sharp, and the circular symmetry is very special and might be beneficial. 

The simulations were repeated with $N=40$ initially non-equidistributed points along the circle, namely half of the points equally distributed along a quarter of the circle and the other half of the points equally distributed along the rest. The time step size was $\delta = 5 \times 10^{-4}$. Figure~\ref{fig:circ_nonequi} on the left displays the initial configuration and the solution at several subsequent times. The shapes look almost circular, however the centre shifts upwards towards the part where the initial mesh was denser. The desired property of better distributing the mesh points can be clearly observed. The solution at the final time $t=10$ is almost stationary and seems stable. 

For comparison, we also made a simulation with the scheme in \cite{DD09} for the geometric flow \eqref{flowDKS} with $\tilde{\lambda} = 1/2$ and the same discretisation parameters and initial data. Figure~\ref{fig:circ_nonequi}, middle displays the shapes at the same times as for the new scheme except for the final time, which here is $t=4.5$. The circle shrinks approximately self-similarly. The mesh points move mostly in the normal direction and also drift very slowly in the tangential direction. However, here in a way that is detrimental to the mesh quality. The vertices in the transition region from the denser to the coarser area move towards the denser area. This leads to instabilities in the longer run and is the reason for the earlier final time. Soon after $t=4.5$ the mesh degenerated and the simulation was aborted.

\subsection{Convergence to a stationary lemniscate}
\label{subsec:lnc}

\begin{figure}
\begin{center}
 \includegraphics[width=6.5cm]{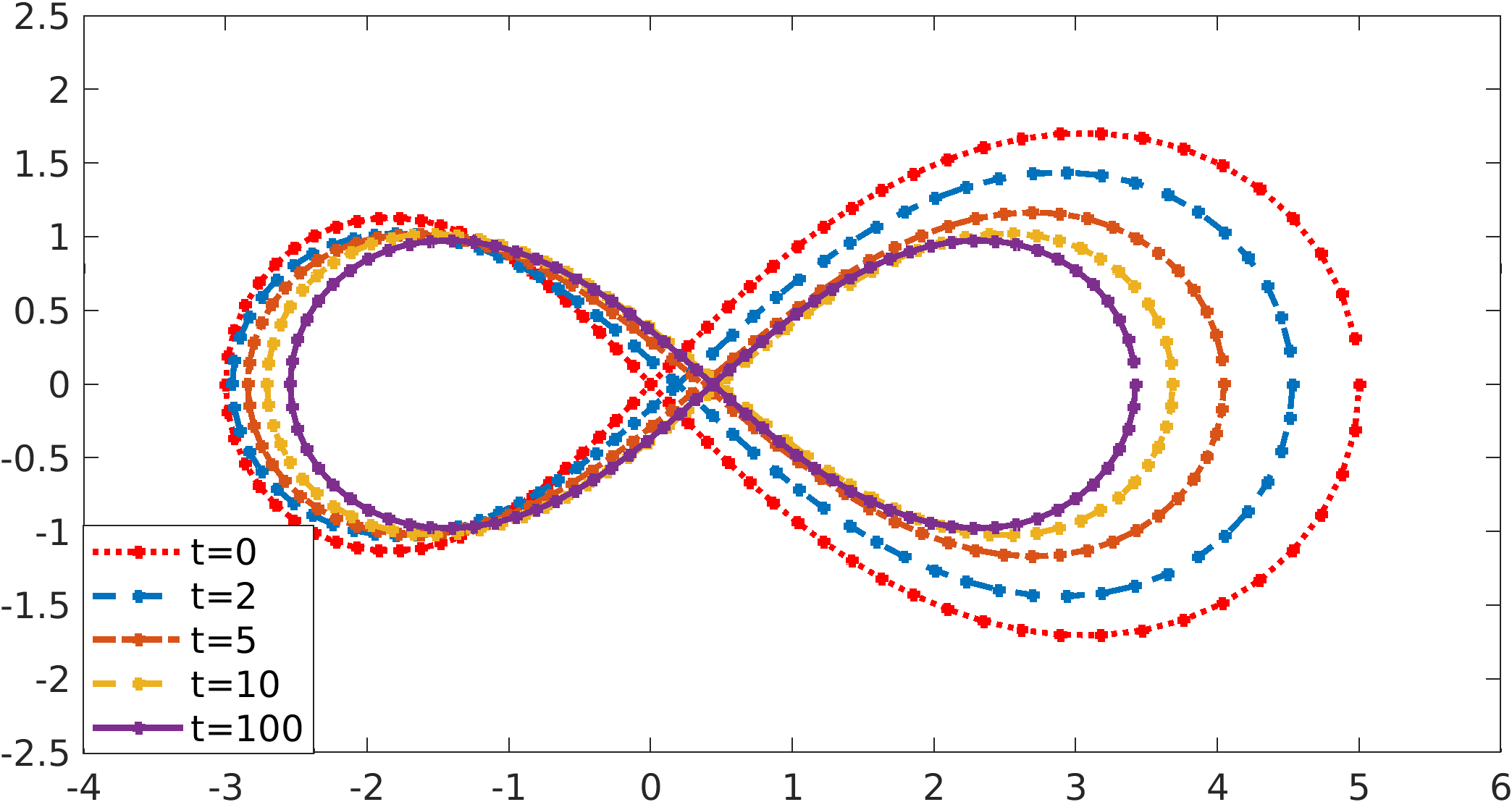} 
 \hfill 
 \includegraphics[width=6.5cm]{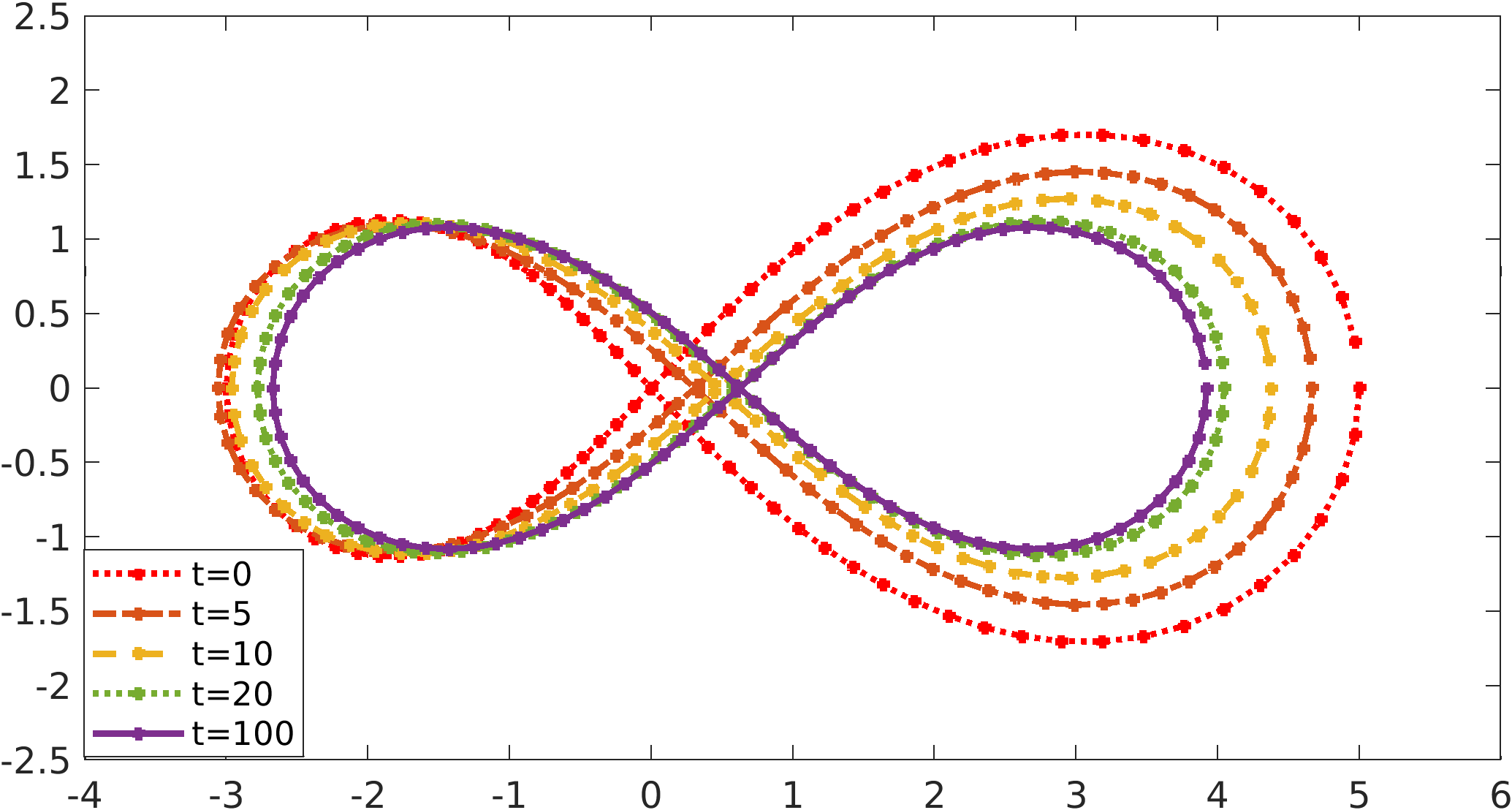} 
 \\
 \includegraphics[width=6.5cm]{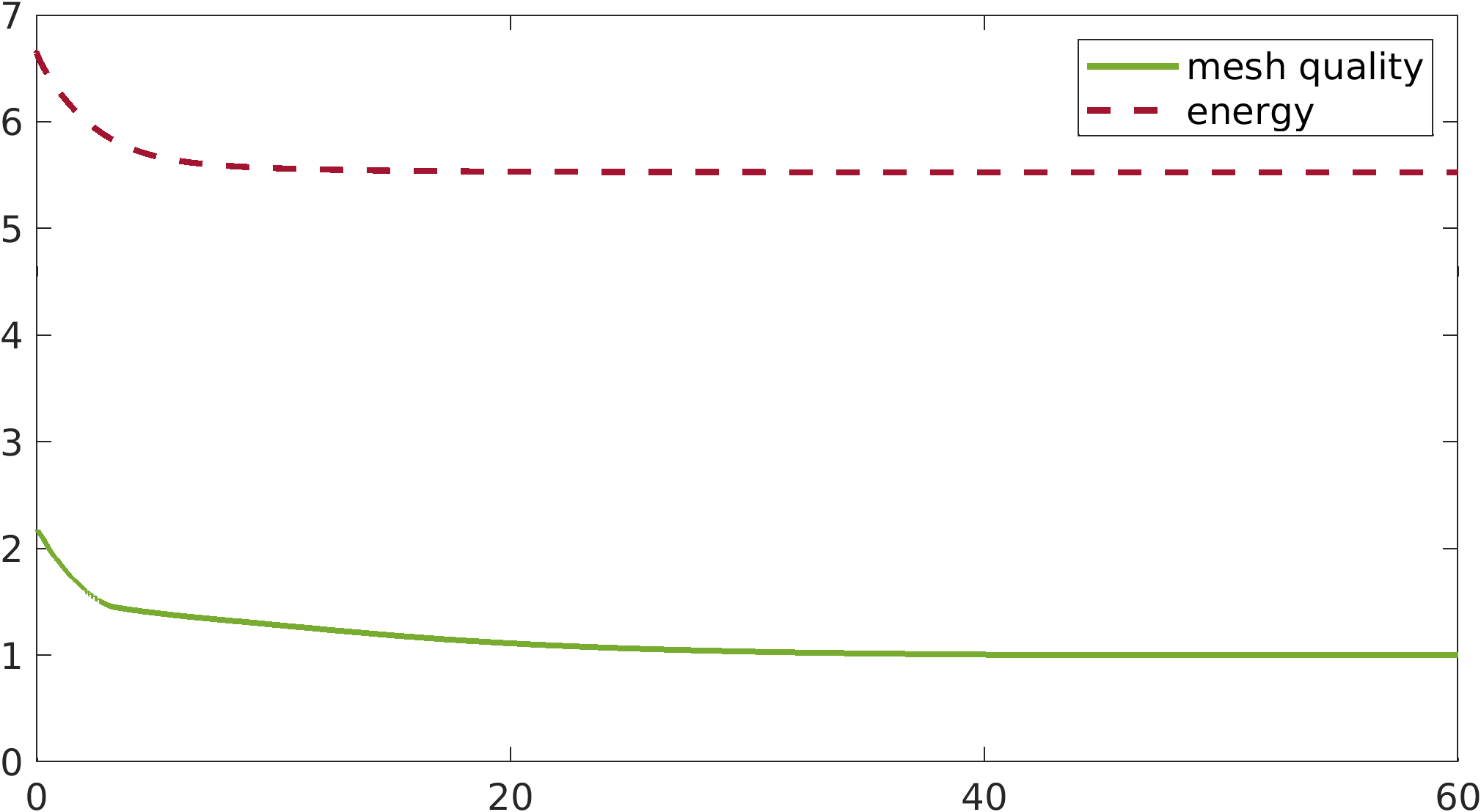} 
 \hfill 
 \includegraphics[width=6.5cm]{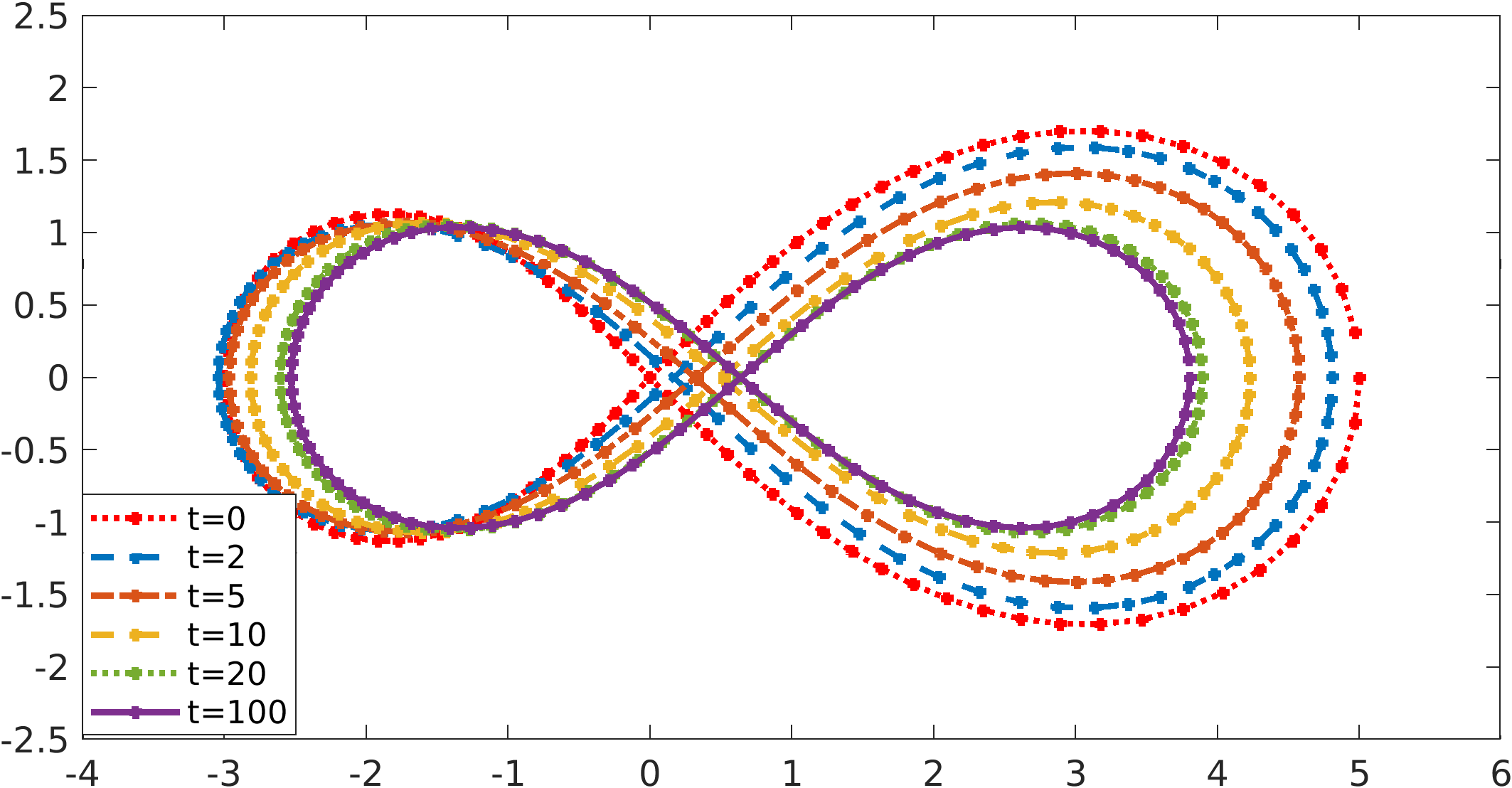} 
\end{center}
\caption{Relaxation of a nonsymmetric lemniscate. Upper left: new scheme \eqref{2.9d}, \eqref{2.10d} with $\lambda = 0.1$ for \eqref{1.2}. Upper right: the scheme in \cite{BGN2007} with $\tilde{\lambda} = 0.2$ for \eqref{flowDKS}. Bottom left: evolution of the energy \eqref{eq:dis_energ} and the mesh quality for the simulation on the top left. Bottom right: extended scheme \eqref{eq:wf1dis}, \eqref{eq:wf2dis} for \eqref{flowDKS} with $\tilde{\lambda} = 0.2$, $\eps = 0.004$, and the choice \eqref{eq:monitor} for $M$. See Sections \ref{subsec:lnc} and \ref{subsec:var} for further detail.}
\label{fig:lnc}
\end{figure}

To further illustrate the benefits of tangential redistribution of mesh points we relaxed an approximate nonsymmetric lemniscate to a stationary state. The initial data were given by 
\[
 x(u) = \frac{(\cos(u) + 4) \cos(u)}{1 + \sin^2(u)} \begin{pmatrix} 1 \\ \sin(u) \end{pmatrix}, 
\]
and we set $\lambda = 0.1$. Simulations were performed with $N=100$ mesh points so that $h = 2\pi / 100 \approx 0.0628$ and a time step size of $\delta = 10^{-3}$. 

Figure \ref{fig:lnc}, top left, displays the initial condition and the evolution. The final shape at time $t=100$ seems stationary. This is supported by the evolution of the discrete energy
\begin{equation} \label{eq:dis_energ}
 \mathcal{D}_{\lambda,h}^{(m)} = \frac{1}{2} \int_{0}^{2\pi} I_{h} (|y_{h}^{(m)}|^2) |x_{hu}^{(m)}| + \lambda |x_{hu}^{(m)}|^2 du, 
\end{equation}
which seems to become constant, see Figure \ref{fig:lnc} at the bottom left. In that graph we also display the evolution of 
\begin{equation} \label{eq:mesh_qual}
 \sigma = \max_{1 \leq j \leq N} q_{j}^{(m)} \Big{/} \min_{1 \leq j \leq N} q_{j}^{(m)}
\end{equation}
as a measure of the mesh quality. Note that a value of one corresponds to equidistribution of mesh points, and this value is approached over time. 

In Figure \ref{fig:lnc}, top right, we provide a comparison with the scheme in \cite{BGN2007}, Sec.~2.3, for \eqref{flowDKS} but with $\tilde{\lambda} = 0.2$ fixed. This higher value for $\tilde{\lambda}$ than for $\lambda$ only serves to ensure that the final, stationary shape is about the same size. The scheme in \cite{BGN2007} also addresses tangential redistribution of mesh points and leads to equidistribution in the long run, as is nicely visible in Figure \ref{fig:lnc}. It satisfies some stability properties but hasn't been proved to converge to the best of our knowledge. This is also the case for subsequent work by these authors such as \cite{BGN2010,BGN12}, which also contain comparisons with the schemes in \cite{DKS} and \cite{DD09} for similar test problems.

\subsection{Hypocycloids, 2D vs 3D}
\label{subsec:hpc}

\begin{figure}
\begin{center}
 \includegraphics[width=3.25cm]{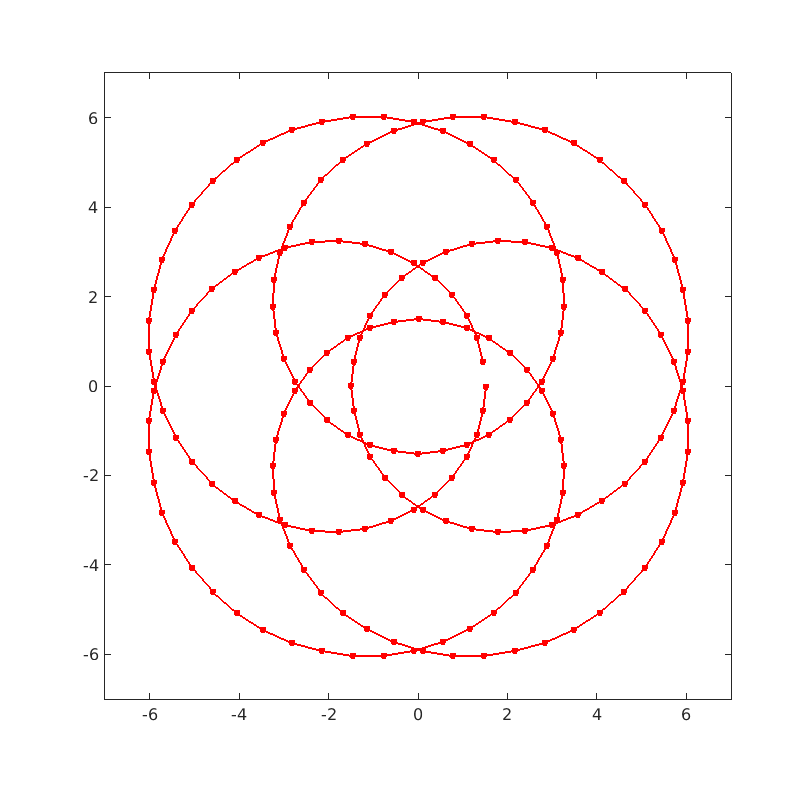} 
 \hfill 
 \includegraphics[width=3.25cm]{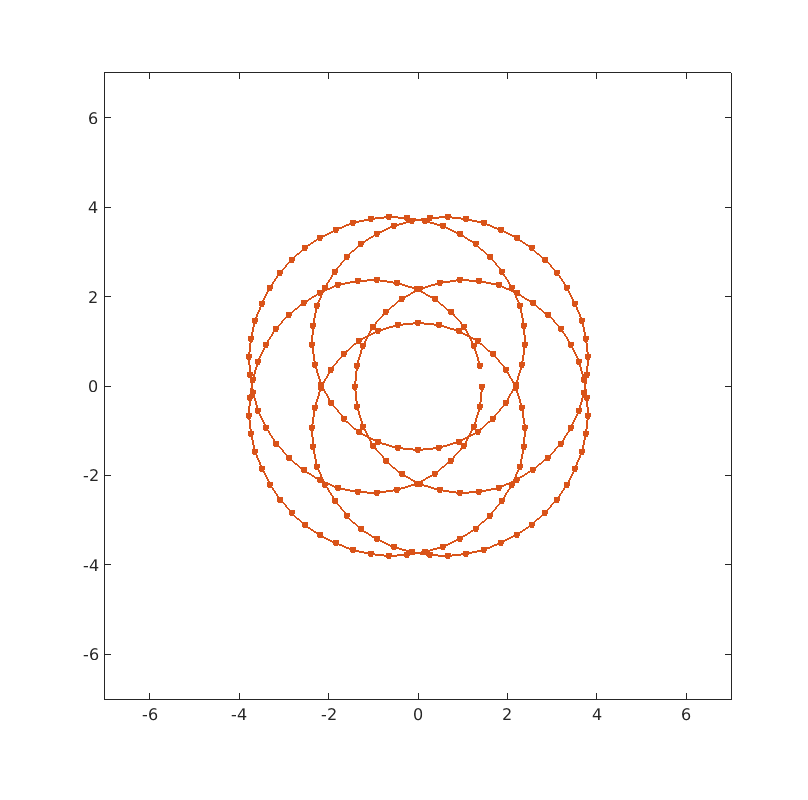} 
 \hfill
 \includegraphics[width=3.25cm]{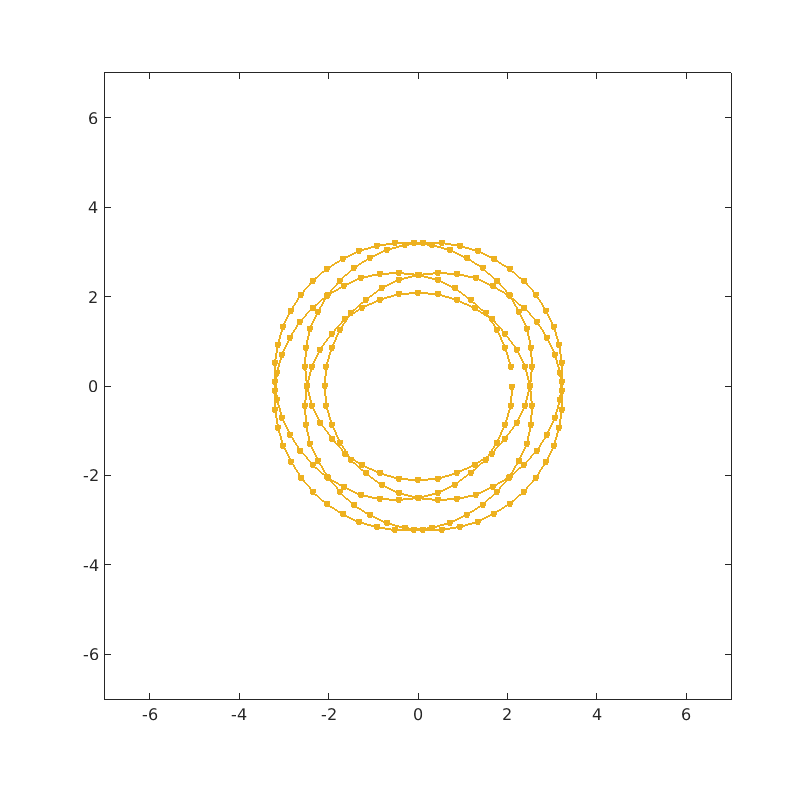} 
 \hfill 
 \includegraphics[width=3.25cm]{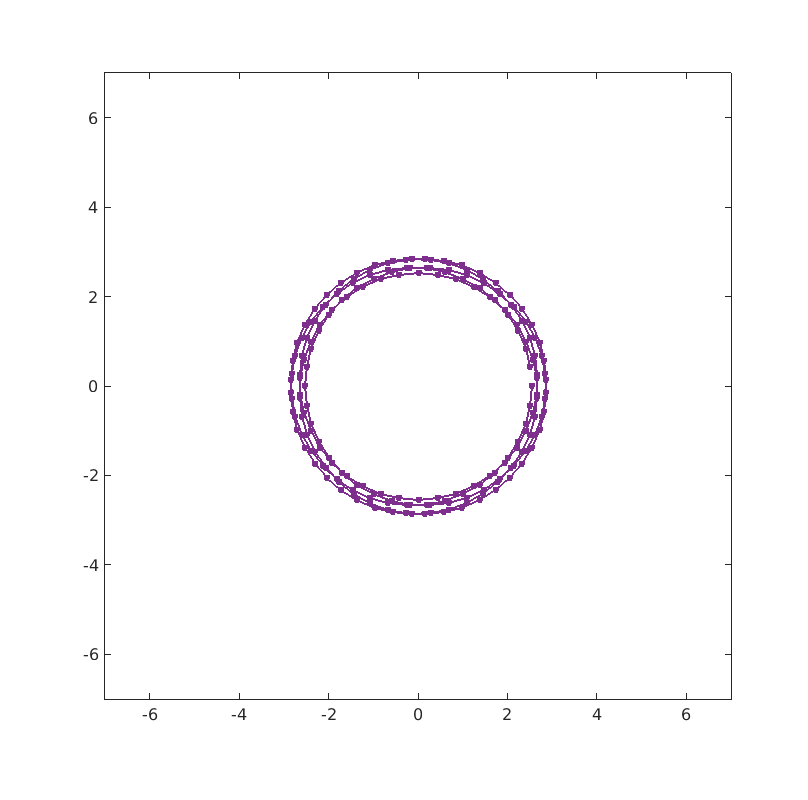} 
\end{center}
\caption{Relaxation of a hypocycloid to a multiply (here five fold) covered circle in 2D. The shape is displayed from the left to the right at times $t=0,600,1500,3000$. See Section \ref{subsec:hpc} for more detail.}
\label{fig:hpc2D}
\end{figure}

\begin{figure}
\begin{center}
 \includegraphics[width=3.25cm]{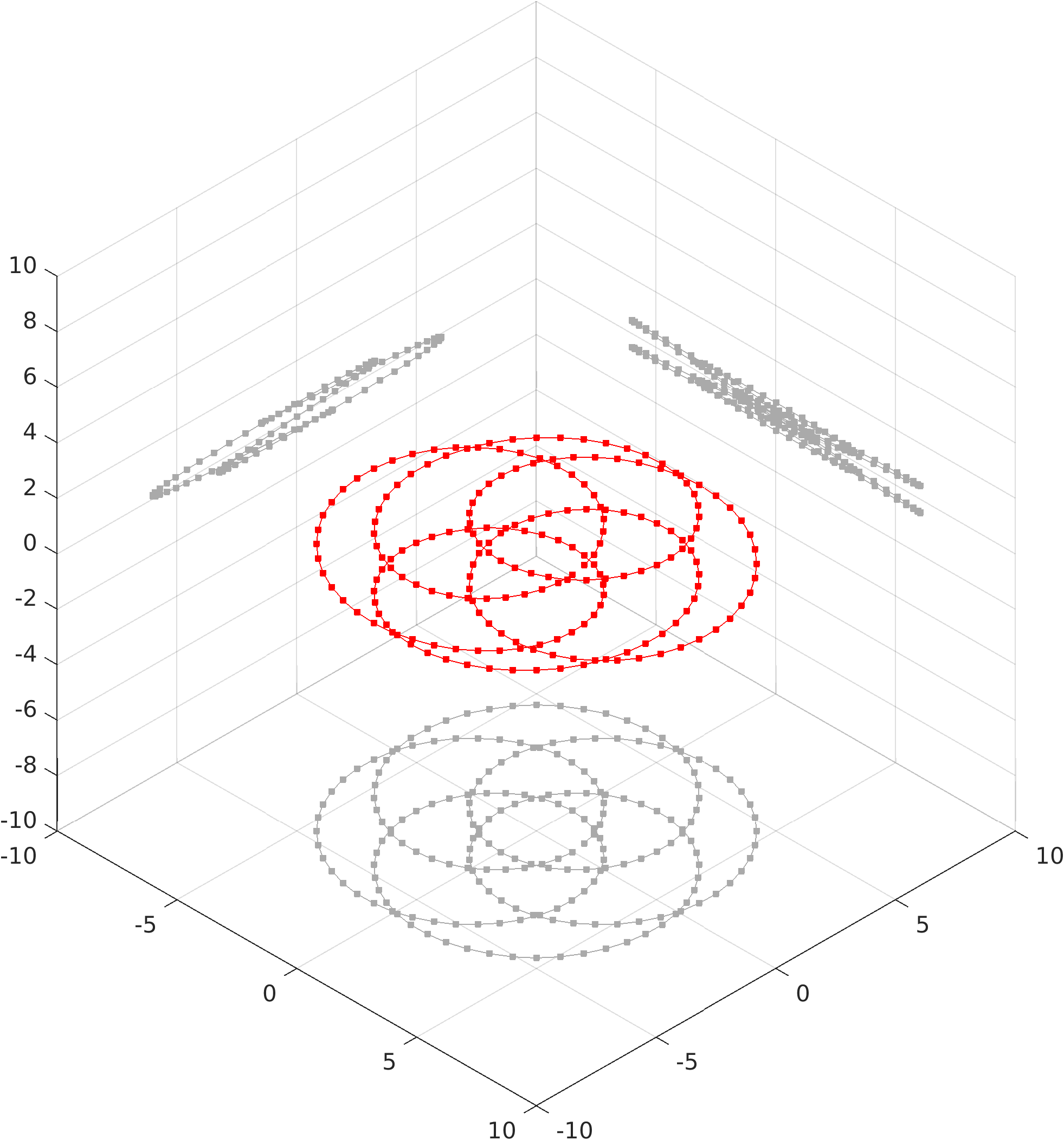} 
 \hfill 
 \includegraphics[width=3.25cm]{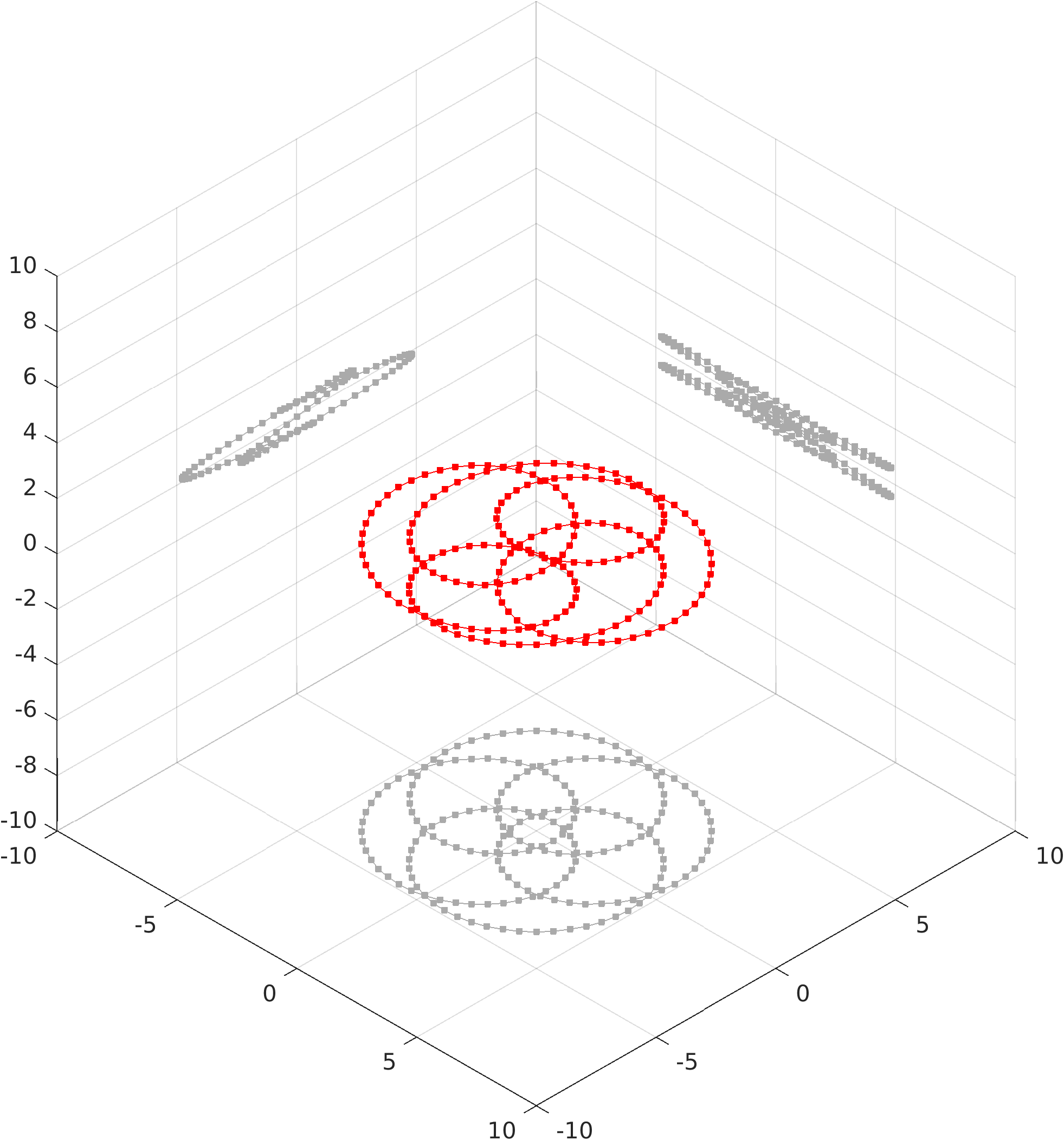} 
 \hfill
 \includegraphics[width=3.25cm]{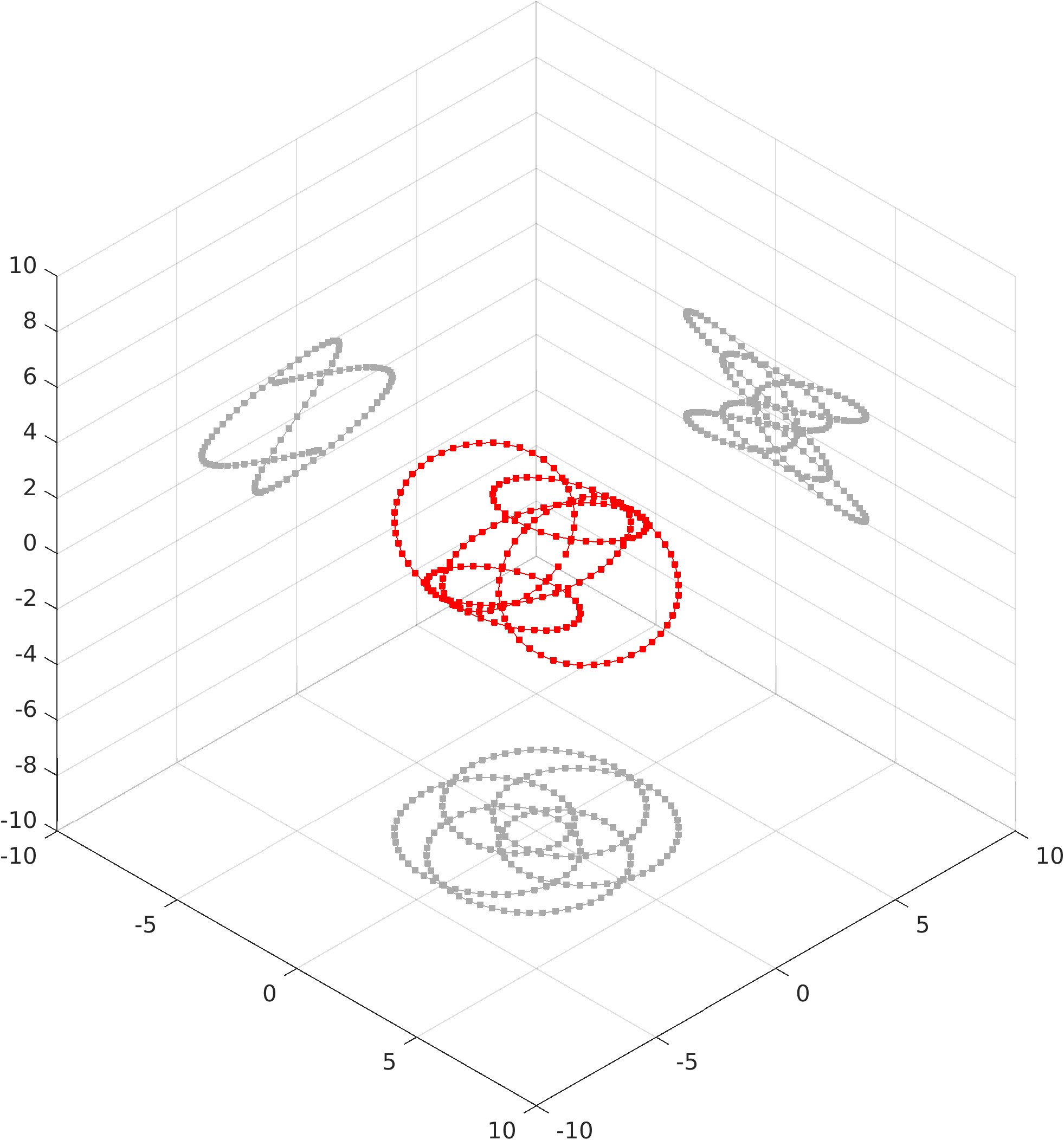} 
 \hfill 
 \includegraphics[width=3.25cm]{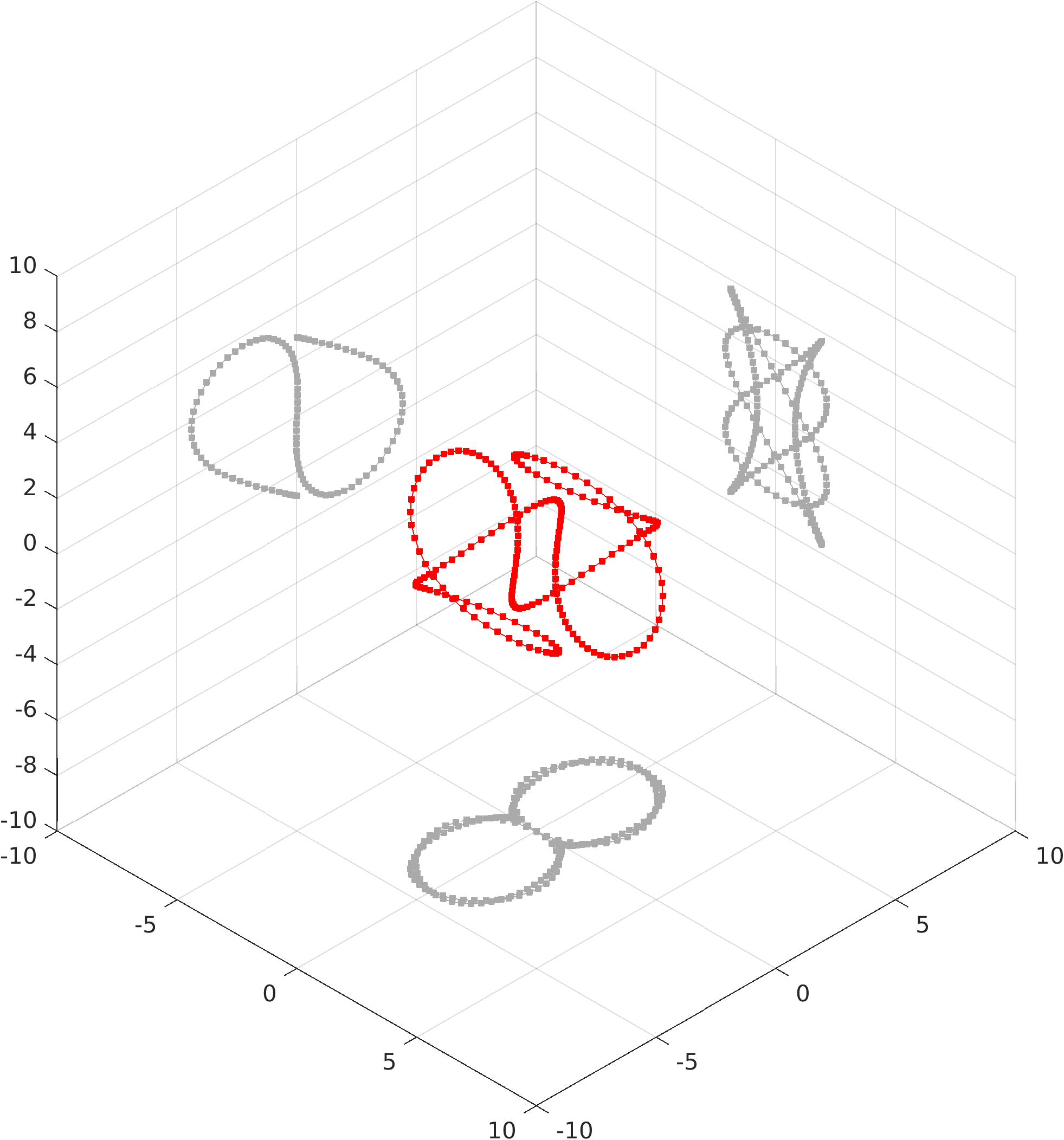} 
 \\
 \includegraphics[width=3.25cm]{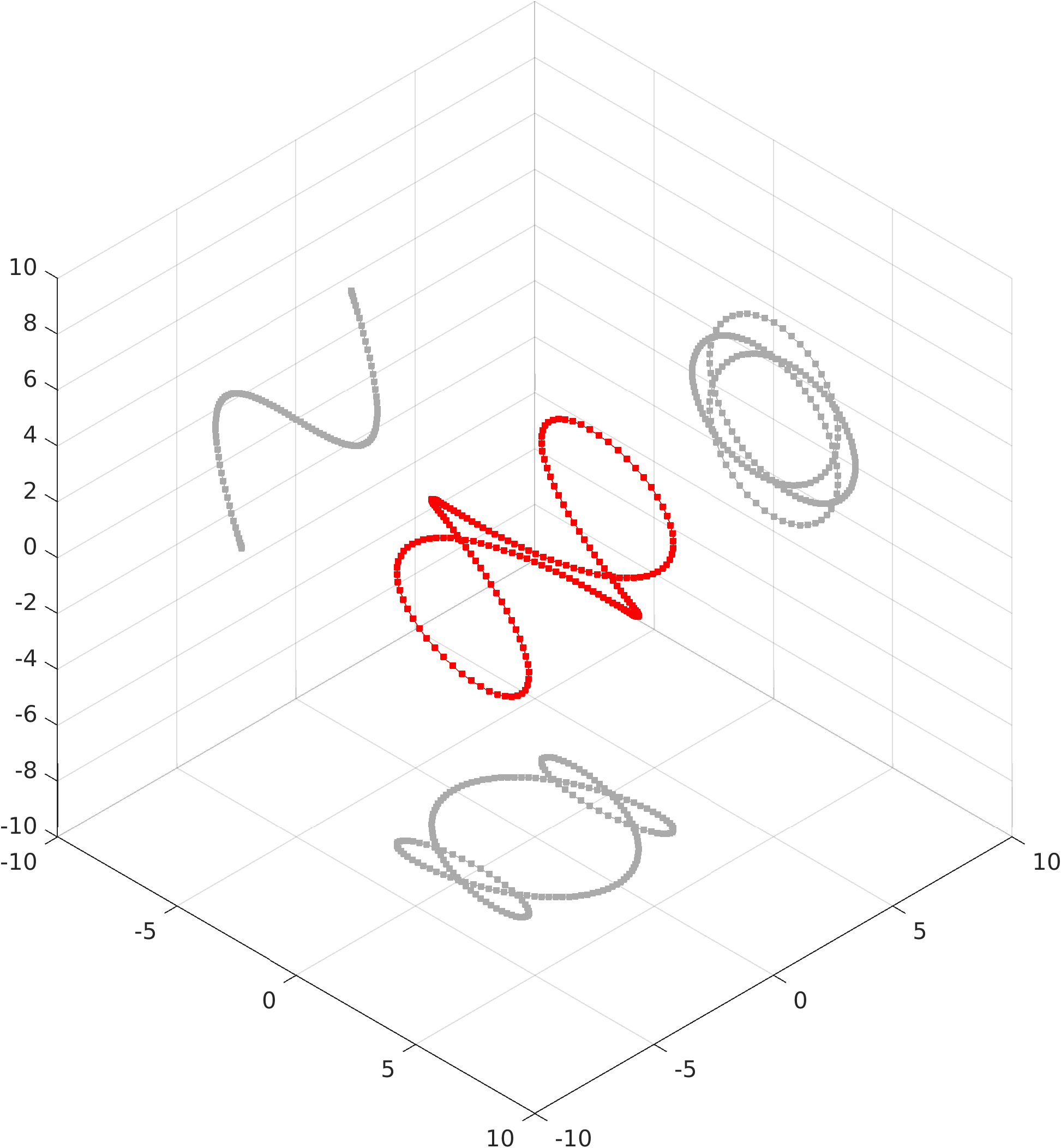} 
 \hfill 
 \includegraphics[width=3.25cm]{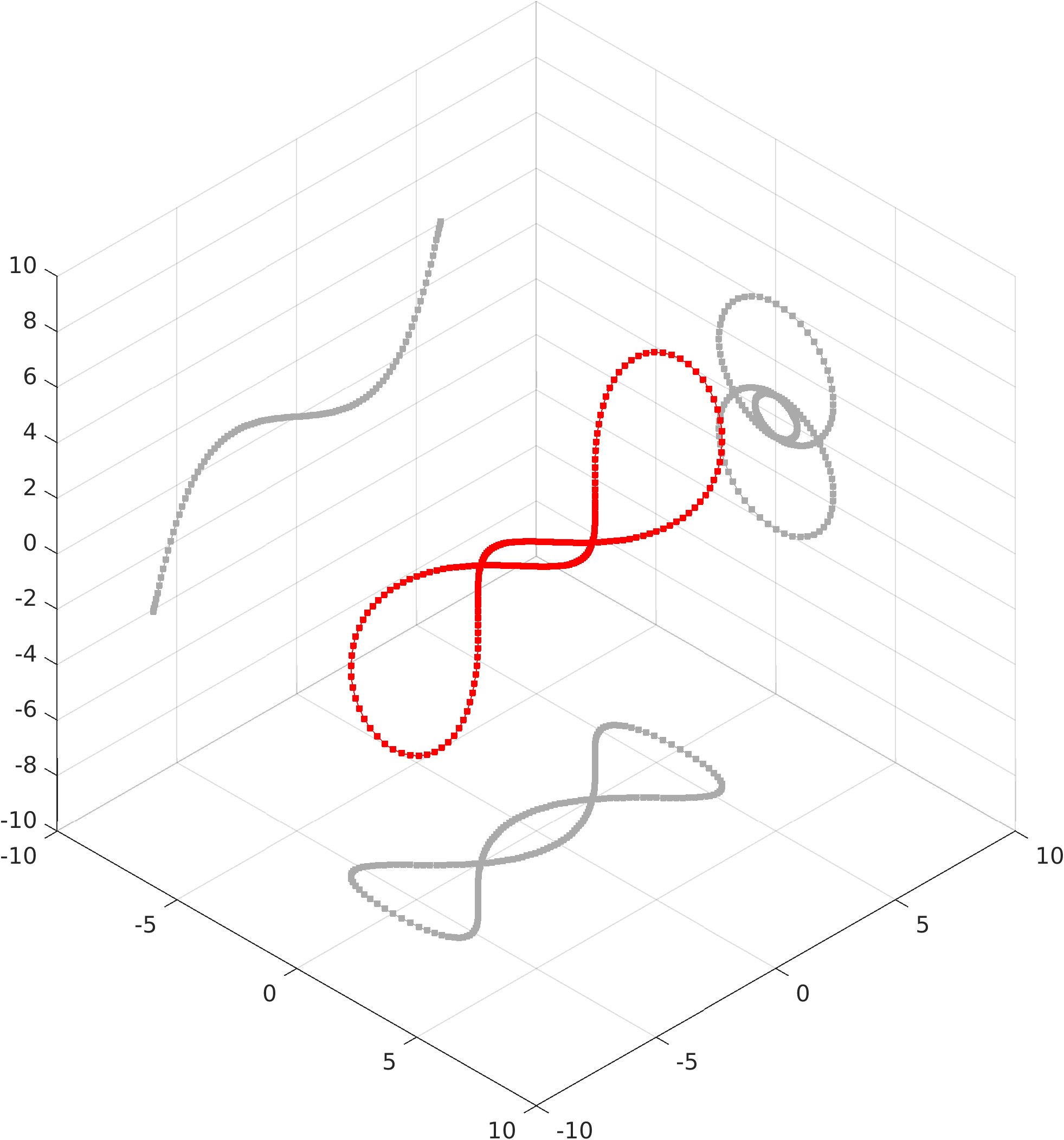} 
 \hfill
 \includegraphics[width=3.25cm]{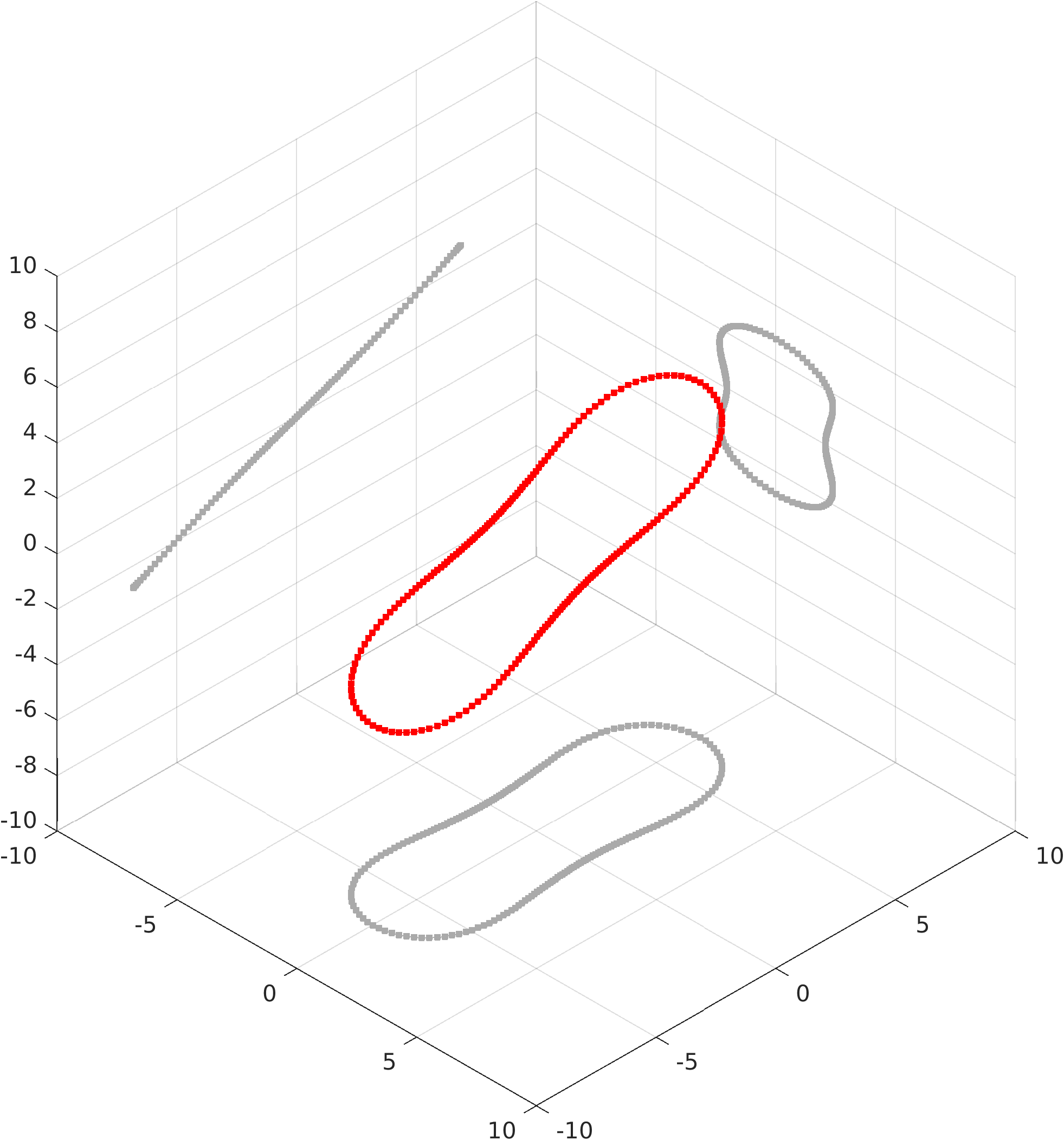} 
 \hfill 
 \includegraphics[width=3.25cm]{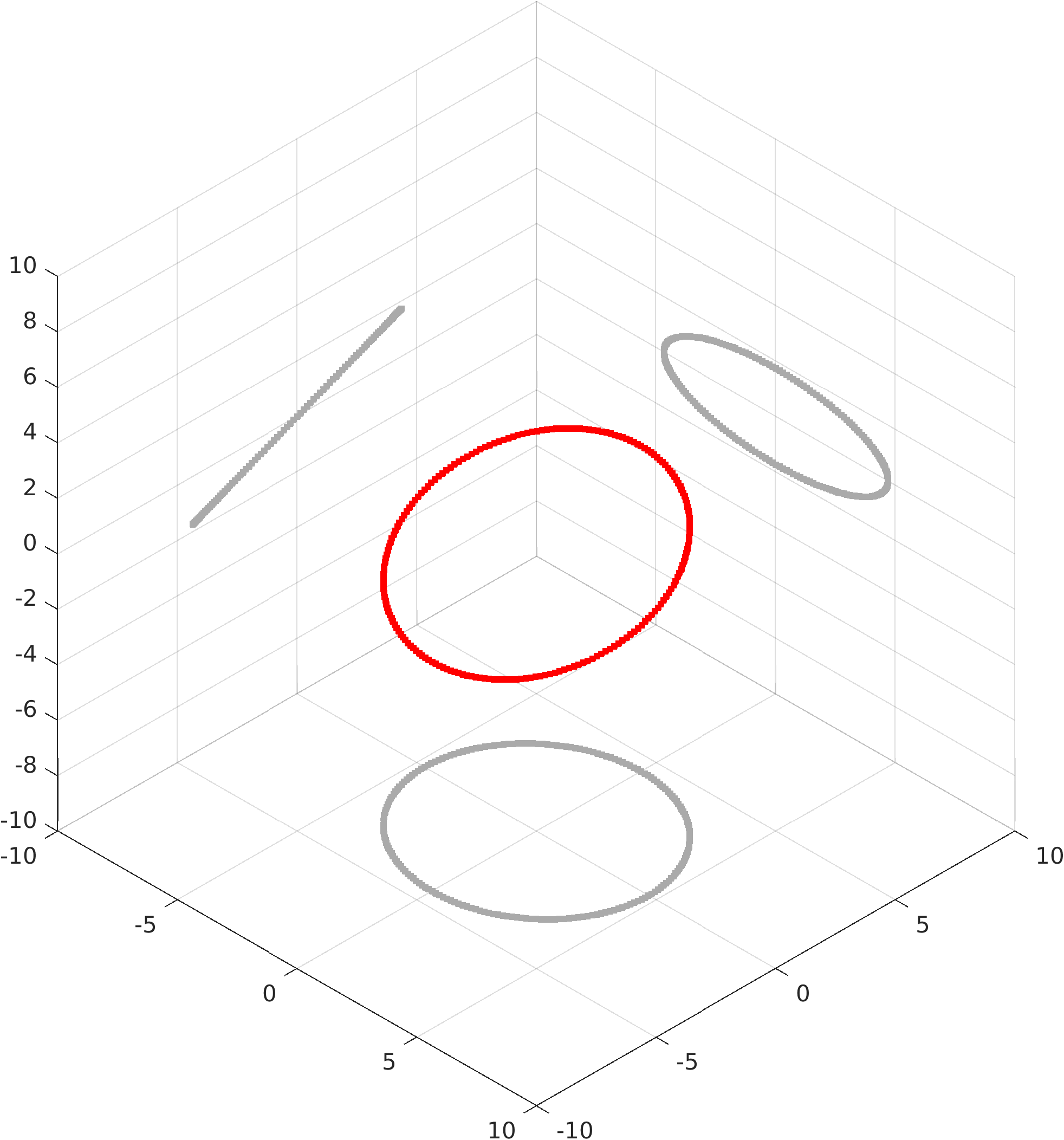} 
 \\
 \includegraphics[width=6.5cm]{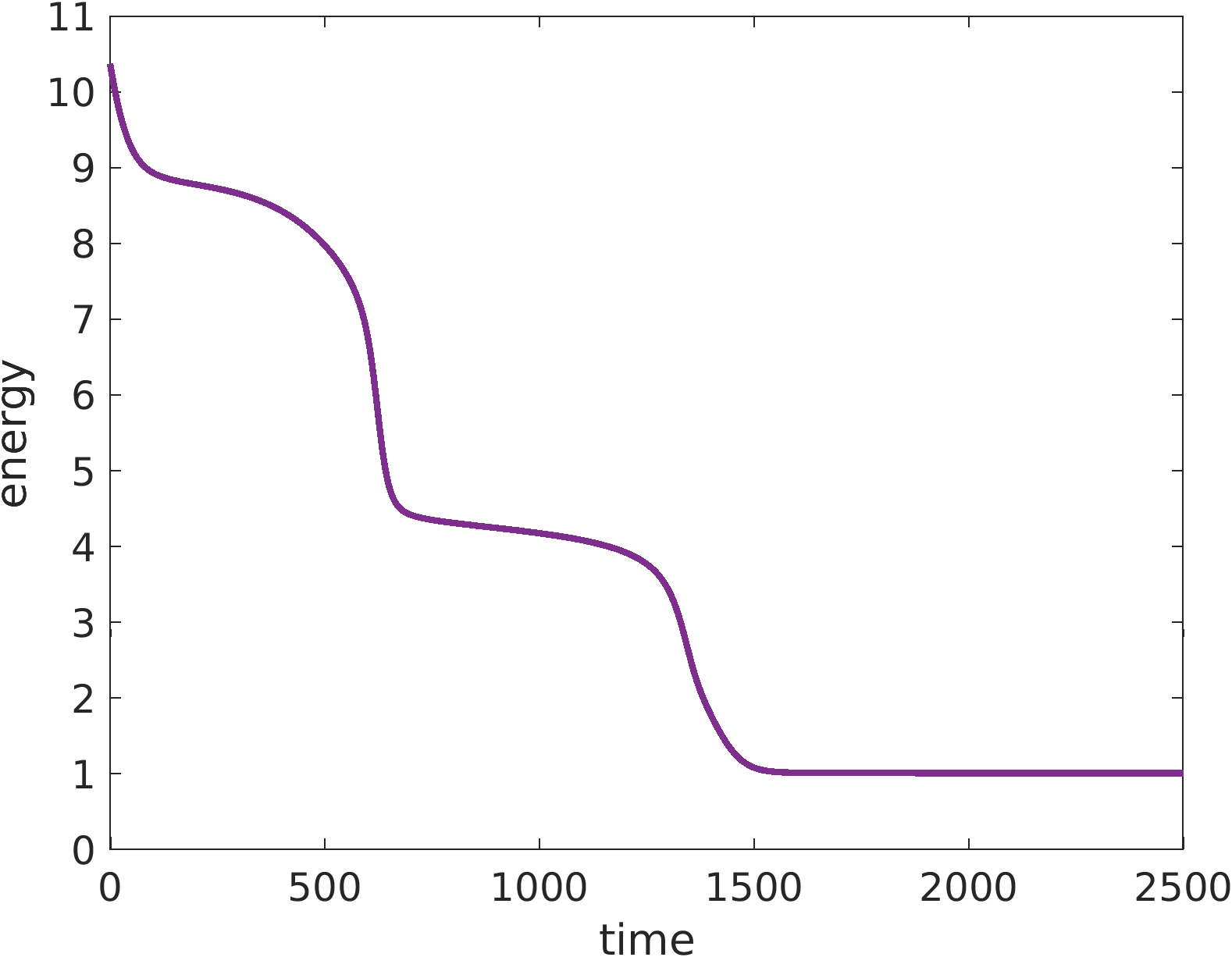} 
 \hfill
 \includegraphics[width=6.5cm]{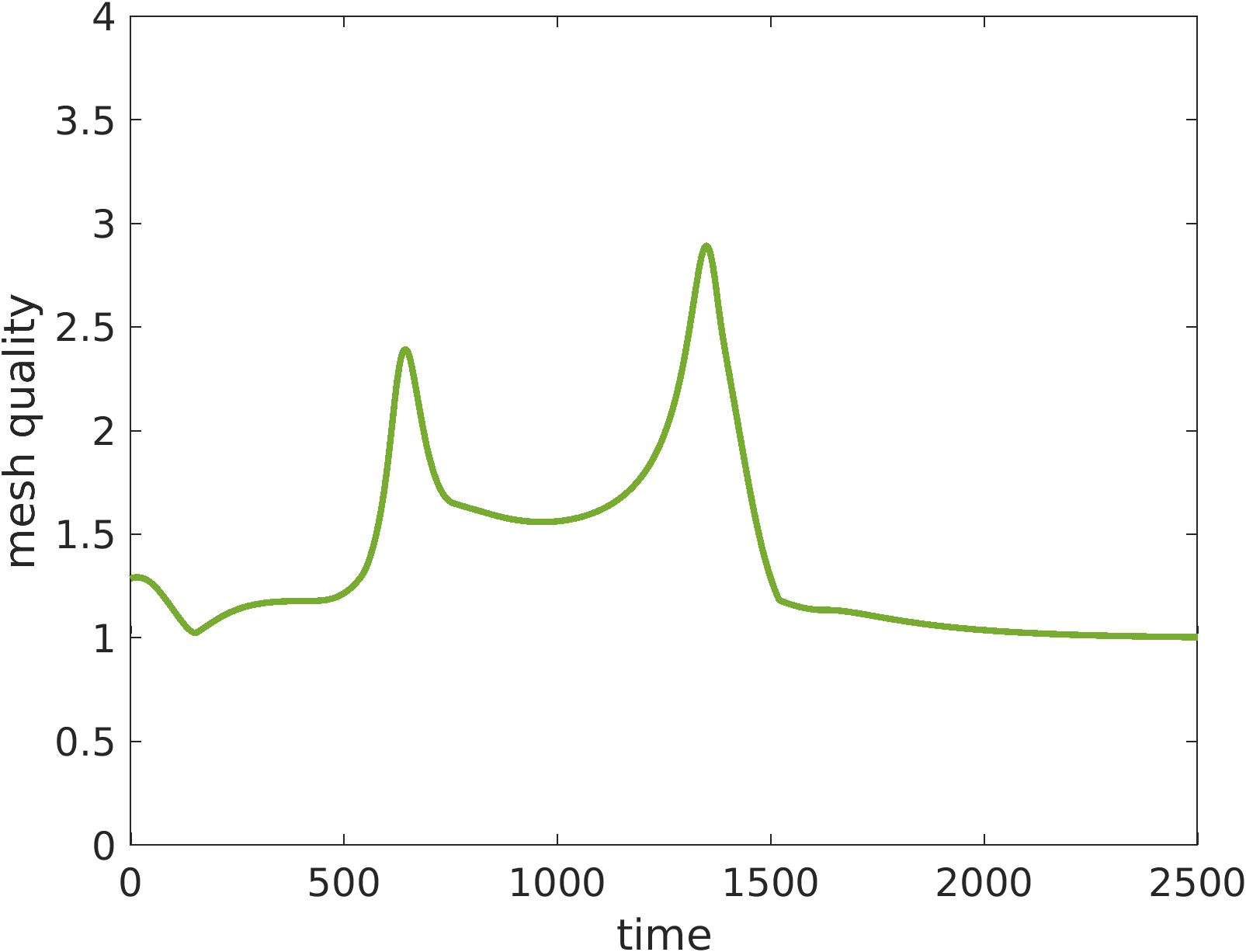} 
\end{center}
\caption{Unreavelling of a perturbed hypocycloid to a simply covered circle in 3D. The shape is displayed in red from the left to the right at times $t=0,100,400,600$ (top) and $t=800,1300,1400,2000$ (middle). The grey curves show the projections to the planes $\{ \xi_1 = -10 \}$, $\{ \xi_2 = 10 \}$, and $\{ \xi_3 = -10 \}$ (where points in $\R^3$ are denoted with $\xi = (\xi_1,\xi_2,\xi_3)$). The bottom row displays the evolution of the discrete energy \eqref{eq:dis_energ} and the mesh quality \eqref{eq:mesh_qual} for this simulation. See Section \ref{subsec:hpc} for more detail.}
\label{fig:hpc3D}
\end{figure}

We performed some computations similarly to Example 4.3 in \cite{DD09} (see also Sec.~5. in \cite{OlRu10}) to illustrate the differences of the flow's behaviour between two and three spatial dimensions. The initial data are given by 
\[
 x_{0}(u) =  \big{(} -\tfrac{5}{2} \cos(u) + 4 \cos(u), -\tfrac{5}{2} \sin(u) + 4 \sin(u), \alpha \sin(3u) \big{)}
\]
for some $\alpha \in \R$. We chose $\lambda = 0.005$ and $N=200$ mesh points. Our time step size was $\delta = 0.05$ and thus significantly larger than in \cite{DD09}. 

In the case $\alpha = 0$ the component in the direction of the third ($z$-) axis remains zero and the evolution takes place in the plane of the first and second coordinate. In the long term, the shape becomes a five-fold covering of a circle. Figure \ref{fig:hpc2D} gives an impression of the evolution. The figure displays a 2D simulation (ignoring the $z$ component), but it has also been performed in 3D. 

However, in the case $\alpha = \tfrac{1}{2}$ this multiply covered circle is not obtained but the curve unravels to form a simple circle. Figure \ref{fig:hpc3D} gives an impression of the evolution in this case. Also, the evolution of the energy \eqref{eq:dis_energ} and the mesh quality \eqref{eq:mesh_qual} are displayed. The energy decays in 'steps' at around $t=650$ and $t=1400$, indicating times of a fairly fast evolution. We can see that the mesh quality parameter $\sigma$ peaks at these times but becomes lower again after and 
approaches one (equidistribution of mesh points) 
towards the end.

\subsection{Extensions}
\label{subsec:var}

We briefly discuss two extensions of our novel method 
and showcase their impact but without analyzing them in detail.
Our focus here is on modifications of the $L^{2}$-gradient flow for the geometric functional \eqref{eq:energylength}.

\begin{enumerate}
 \item 
 Picking up an idea from \cite{PozSti2021} 
 (which is on planar curves)
 we may split up the velocity in its normal and tangential part, $x_{t} = P x_{t} + (x_{t} \cdot \tau) \tau$ and then scale all terms with tangential contributions with a small parameters $0 < \eps \ll 1$. Equivalently, we consider a weighted $L^2$ gradient flow 
 for
\begin{align}\label{ELD}
 \mathcal{E}(x) + \tilde{\lambda} \mathcal{L}(x) + \epsilon \mathcal{D}(x), 
\end{align}
 where the tangential part has an $\eps$ weighting. 
 If we do this then we obtain (see also \eqref{eq:wf1}, \eqref{eq:wf2} below with $M=1$) 
 \[
  P x_{t} + \eps (x_{t} \cdot \tau) \tau = -\nabla_{s}^{2} \vec{\kappa} - \frac{1}{2} |\vec{\kappa}|^{2} \vec{\kappa}  + \tilde{ \lambda} \vec{\kappa}  + \eps   \frac{x_{uu}}{|x_{u}|} . 
 \]
 In effect, the tangential movement of mesh points that is beneficial for the mesh quality and due to the Dirichlet energy contribution in \eqref{newFun} is kept. 
 We have not analyzed this idea but refer to \cite{PozSti2021}, where a network of curves moving by curve shortening flow is investigated. Let us note that, there, constants in error estimates depend in an unfavorable way on $\eps$, as does the conditioning of the linear systems in the fully discrete scheme. 
 We expect this to be the case for the scheme \eqref{eq:wf1dis}, \eqref{eq:wf2dis} (with $M=1$) for the above approach, too.
 
 \item 
 Whilst equi-distribution of mesh points along the curve might generally be desirable, some applications might benefit from more mesh points in certain areas of interest. These could simply be parts of the curve where the curvature is high. Or, when combining the geometric equation with a reaction-diffusion system on the curve to model cell migration (for instance, see \cite{EllStiVen2012}), then the fields on the curve might require a better resolution in some parts. \\
 Drawing on ideas presented in \cite{MacNolRowIns2019} 
 we introduce a monitor function that acts as a weighting in the term arising from the Dirichlet energy. Intuitively, it may be interpreted as kind of a tension parameter. The higher $M$ the higher the attraction between close points. In the discrete setting one will therefore expect that vertices move the closer the higher $M$ is. 
 
 For simplicity, we here only consider a (constant in time) function defined on the ambient space $M : \R^{n} \to (0,\infty)$. 
 
\end{enumerate}

Combinding these ideas, 
the weak formulation reads 
\begin{align} 
 \int_{0}^{2\pi} \big{(} P x_{t} \cdot \phi + \eps (x_{t} \cdot \tau) \tau \cdot \phi \big{)} |x_{u}| - \int_{0}^{2\pi} \Big{(} \frac{P y_{u} \cdot \phi_{u}}{|x_{u}|} + \frac{|y|^{2}}{2} (\tau \cdot \phi_{u} ) \Big{)} \nonumber \\
 + \tilde{\lambda} \int_{0}^{2\pi} \tau \cdot \phi_{u} + \eps \int_{0}^{2\pi} M(x) x_{u} \cdot \phi_{u} &=0, \label{eq:wf1} \\
 \int_{0}^{2\pi} (y \cdot \psi) |x_{u}| + \int_{0}^{2\pi} (\tau \cdot \psi_{u} ) &=0, \label{eq:wf2}
\end{align}
for all $\phi, \psi \in H^{1}_{per}(0, 2\pi)$ and $t \in [0,T]$. The corresponding strong PDE can be split up in normal and tangential part in the form 
\begin{align} \label{eq:weightedflow}
 Px_{t} &= -\nabla_{s}^{2} \vec{\kappa} - \frac{1}{2} |\vec{\kappa}|^{2} \vec{\kappa} + \tilde{\lambda} \vec{\kappa} + \eps P \big{(} M x_{u} \big{)}_{u} \frac{1}{|x_{u}|}, \\
 x_{t} \cdot \tau &= \frac{1}{|x_{u}|} \big{(} M x_{u} \big{)}_{u} \cdot \tau. 
\end{align}
Thus formally as $\epsilon \to 0$ we recover \eqref{gflow}. 
Moreover note that, in general, the above scheme is not a gradient flow due to the term with the monitor function $M$.
We performed some simulations with the following approximation scheme: 

\begin{algo}
Set $x^{(0)} = I_{h} x_{0}$ and successively determine $x_{h}^{(m+1)}$, $y_{h}^{(m+1)} \in X_{h}^{n}$ for $m=0,1,2, \ldots,m_{T}-1$ by solving the linear problem  
\begin{align} 
\int_{0}^{2\pi} \left ( P_{h}^{(m)} \frac{x_{h}^{(m+1)} - x_{h}^{(m)}}{\delta} \cdot \phi_{h} + \eps \Big{(} \frac{x_{h}^{(m+1)} - x_{h}^{(m)}}{\delta} \cdot \tau_{h}^{(m)} \Big{)} \tau_{h}^{(m)} \cdot \phi_{h} \right ) |x_{hu}^{(m)}|  \nonumber \\
- \int_{0}^{2\pi} \left ( \frac{P_{h}^{(m)} y_{hu}^{(m+1)} \cdot \phi_{hu}}{|x_{hu}^{(m)}|} + \frac{1}{2} I_{h}(|y_{h}^{(m)}|^{2}) \Big{(} \frac{x_{hu}^{(m+1)}}{|x_{hu}^{(m)}|} \cdot \phi_{hu} \Big{)} \right ) \nonumber \\
+ \tilde{\lambda} \int_{0}^{2\pi} \frac{x_{hu}^{(m+1)}}{|x_{hu}^{(m)}|} \cdot \phi_{hu} + \eps \int_{0}^{2\pi} I_{h} (M) x_{hu}^{(m+1)} \cdot \phi_{hu} =&0, \label{eq:wf1dis} \\
\int_{0}^{2\pi} I_{h}(y_{h}^{(m+1)} \cdot \psi_{h}) |x_{hu}^{(m)}| + \int_{0}^{2\pi} \Big{(} \frac{x_{hu}^{(m+1)}}{|x_{hu}^{(m)}|} \cdot \psi_{hu} \Big{)} =&0, \label{eq:wf2dis} 
\end{align}
for all test functions $\varphi_{h}, \psi_{h} \in X_{h}^{n}$. 
\end{algo}

We repeated the simulation of the circle with initially non-equidistributed vertices from Section~\ref{subsec:circ} with the same discretization parameters. In addition, we set $M = 1$ and $\eps = 10^{-2}$. The result is displayed in Figure~\ref{fig:circ_nonequi} on the right. 

We observe lateral redistribution of the mesh points similar to result for our new scheme (Figure~\ref{fig:circ_nonequi}, left). In turn, the evolution of the shape and position looks closer to the result from scheme in \cite{DD09} (Figure~\ref{fig:circ_nonequi}, middle). 

\medskip

We also relaxed the nonsymmetric lemniscate from Section~\ref{subsec:lnc} with $\tilde{\lambda} = 0.2$. Here, we set $\eps = 0.004$ and chose 
\begin{equation} \label{eq:monitor}
 M: \R^2 \to (0,\infty), \quad M(\xi_1,\xi_2) = 1 + \frac{(\xi_1-1)^2}{10}. 
\end{equation}
This encourages higher vertex density the further a part of the curve is away from the line $\{ \xi_1 = 1 \}$. The number of vertices $N=100$ was the same but we had to choose a significantly smaller time step size of $\delta = 10^{-5}$ for stability reasons. 

The evolution is displayed in Figure~\ref{fig:lnc}, bottom right. In the final shape at time $t=100$ a lower density of vertices around the intersection point indeed is noticeable, whilst the density looks a bit higher in the curved parts away from the line $\{ \xi_1 = 1 \}$. We also note that the final shape is slightly smaller than the one obtained with the scheme in \cite{BGN2007} (see Figure~\ref{fig:lnc}, top right). However, simulations with several values of $\eps$ revealed that the shapes get closer as $\eps \searrow 0$.

\section{Conclusion}

The elastic flow can be amended in various ways to prevent solutions from continuous growing, adding  the Dirichlet energy as proposed in \eqref{newFun} is one of them. However, this approach has a significant advantage for finite element computations. First, it involves a tangential movement of mesh points of great benefit to the mesh quality during evolution. Second, it provides additional control so that the error due to the finite element discretization can be estimated. We provided such estimates in Theorem \ref{mainteo} in norms that are  natural for the problem. These results were supported by numerical simulations, which also displayed good mesh qualities throughout the evolution. 

The resulting flow is no longer geometric, however, in the sense that the right-hand side of \eqref{1.2} depends on the parametrization of the curves (as opposed to \eqref{flowDKS} or \eqref{gflow}). As an alternative, we briefly looked into extensions of \eqref{gflow} for which the addition of a Dirichlet functional plays the role of an $\epsilon$ perturbation term.
Here some remaining open questions are related 
to the time discretization,  the impact of the small parameter~$\epsilon$ and a monitor function involved in the extensions.

\end{document}